\documentclass[11pt]{amsart}
\usepackage{
  geometry,
  amsrefs,
  amsthm,
  mathtools,
  amssymb,
  graphicx,
  epstopdf,
  color,
  cancel,
  multirow,
  arydshln
}
\usepackage[colorlinks=true,linkcolor=blue,urlcolor=cyan,citecolor=blue]{hyperref}
\mathtoolsset{showonlyrefs}

\newtheorem{theorem}{Theorem}
\newtheorem*{theorem*}{Theorem} % For duplicated theorem statements in introduction
\newtheorem{prop}{Proposition}
\newtheorem{lem}{Lemma}
\newtheorem{cor}{Corollary}
\newtheorem{remark}{Remark}

\renewcommand{\Re}{\operatorname{Re}}
\renewcommand{\Im}{\operatorname{Im}}
\newcommand{\tr}{\operatorname{tr}}
\newcommand{\I}{\operatorname{i}}
\newcommand{\e}{\operatorname{e}}
\newcommand{\id}{\operatorname{I}}

\newcommand{\hBlock}{H}
\newcommand{\reeb}{\mathcal{Z}}

\title{Canonical Bochner--K\"{a}hler potentials via the  Sasaki--K\"{a}hler correspondence}

\makeatletter
\def\@tocline#1#2#3#4#5#6#7{\relax
  \ifnum #1>\c@tocdepth % then omit
  \else
    \par \addpenalty\@secpenalty\addvspace{#2}%
    \begingroup \hyphenpenalty\@M
    \@ifempty{#4}{%
      \@tempdima\csname r@tocindent\number#1\endcsname\relax
    }{%
      \@tempdima#4\relax
    }%
    \parindent\z@ \leftskip#3\relax \advance\leftskip\@tempdima\relax
    \rightskip\@pnumwidth plus4em \parfillskip-\@pnumwidth
    #5\leavevmode\hskip-\@tempdima
      \ifcase #1
       \or\or \hskip 1em \or \hskip 2em \else \hskip 3em \fi%
      #6\nobreak\relax
    \dotfill\hbox to\@pnumwidth{\@tocpagenum{#7}}\par
    \nobreak
    \endgroup
  \fi}
\makeatother

\begin{document}

%%%%%%%%%%%%%%%%%%%%%%%%% more metadata
\author[Martin Kol\'{a}\v{r}]{Martin Kol\'{a}\v{r}}
\address{Martin Kol\'{a}\v{r},
	Department of Mathematics and Statistics 
	Masaryk University,
	Kotl\'{a}\v{r}sk\'{a}~2,
	611 37 Brno,
	Czech Republic. ORCID ID: 0000-0003-4514-5636}\email{ mkolar@math.muni.cz}
\author[Gerd Schmalz]{Gerd Schmalz}
\address{Gerd Schmalz,
School of Science and Technology, University of New England, Armidale, 2351, Australia. ORCID ID: 0000-0002-6141-9329}
    \email{schmalz@une.edu.au}
 \author[David Sykes]{David Sykes}
\address{David Sykes,
	Institute for Basic Science, Center for Complex Geometry, Daejeon, 34126, Republic of Korea. ORCID ID: 0000-0002-6928-3753}\email{ sykes@ibs.re.kr}

\begin{abstract}
    We introduce a new approach to the study of Bochner--K\"{a}hler manifolds, based on normal form techniques in CR geometry. As observed by Webster, Bochner--K\"{a}hler manifolds are intimately related to spherical Sasakian manifolds. Our approach utilises this relationship as well as normal forms for spherical Sasakian manifolds (known as rigid spheres). In this setting potentials of Bochner--K\"{a}hler manifolds are nothing but the defining equations of rigid spheres in rigid normal form.

This leads to a canonical formula for Bochner--K\"{a}hler potentials, as well as streamlined proofs of known results on the structure and the symmetries of Bochner--K\"{a}hler manifolds. Our approach also provides a host of examples in terms of an implicit formula of the corresponding K\"{a}hler potentials.
\end{abstract}

%%%%%%%%%%%%%%%%%%%%%%%%% body
\maketitle
\tableofcontents

\section{Introduction}
 Bochner--K\"{a}hler manifolds appear as a natural complex analog to conformally flat manifolds in Riemannian Geometry. 
In analogy to the splitting of the Riemann curvature into the Ricci curvature and the traceless Weyl curvature on Riemannian manifolds, the Riemann curvature on a K\"{a}hler manifold splits into the Ricci part and the traceless Bochner curvature. Bochner--K\"{a}hler manifolds are K\"{a}hler manifolds with vanishing Bochner curvature. Such manifolds have been studied by many authors, with  foundational theory by Bryant developed in \cite{MR1824987}. Among other results, Bryant described the moduli space of Bochner--K\"{a}hler structures, provided a series of examples, and determined possible symmetry algebra dimensions, proving the presence of a large symmetry algebra in particular. Bryant's approach features techniques of Cartan applied to analysing structure equations of canonical differential forms on the K\"{a}hler manifold's unitary coframe bundle.

Webster \cite{MR0520599} had earlier discovered a correspondence between Bochner--K\"{a}hler manifolds and CR-flat Sasakian structures, providing another viewpoint through which to study the K\"{a}hler geometry. This approach has been applied by David and Gauduchon \cite{MR2237108}, Pan\'ak and Schwachh\"{o}fer\cite{MR2462805}, Bolsinov and Rosemann \cite{MR4258120}, and others.  In the sequel, we apply the Webster correspondence from a CR geometric point of view to derive both new results and streamlined proofs of known theorems on Bochner--K\"{a}hler geometries, and wish to emphasise the relative simplicity in arguments that this viewpoint facilitates.

We apply a fundamentally different,  analytic approach based on normal form techniques, originating in the work of Chern and Moser. The main point is to establish the equivalence  of describing completely all K\"{a}hler  potentials for  Bochner--K\"{a}hler structures and a purely CR geometric problem, describing so called rigid spheres. Stanton in \cite{St91} posed this latter question and gave a list of rigid spheres, without proving completeness. Later, Ezhov, Schmalz and Kol\'a\v{r} in \cite{ES15,MR3833790} showed that Stanton's list is not complete and suggested a way to complete it, which provides the starting point for our approach.

In particular, using this approach, we derive the following succinct general description of the K\"{a}hler potentials defining Bochner--K\"{a}hler structures.
\newtheorem*{thmA}{\bf Theorem \ref{thm: general potentials}}
\begin{thmA}
    A real-valued function $F$ on a complex manifold $M$ is a K\"{a}hler potential of a Bochner--K\"{a}hler structure in a neighbourhood of a point $p\in M$ if and only if
    \begin{equation}\label{eqn: potential intro}
        \begin{bmatrix} 0& z^* & \frac{\I}{2} \end{bmatrix} \e^{2\I AF}\begin{bmatrix} 1\\ z\\ 0 \end{bmatrix}=0
    \end{equation}
    in some coordinates $z=(z_1,\ldots,z_n)$ centered at $p$ for some matrix $A\in \mathfrak{gl}_{n+2}(\mathbb{C})$ of the form 
    \begin{equation}\label{eqn: potential parameters intro}   
    A=\begin{bmatrix}
        0& -2\I a^* & r\\0&\I X& a\\1&0&0
    \end{bmatrix},
    \end{equation}
    where 
    \begin{itemize}
        \item $X\in \mathfrak{gl}_{n}(\mathbb{R})$ is diagonal with entries ordered such that $X_{j,j}\geq X_{j+1,j+1}$,
        \item $a$ is a vector with non-negative real entries satisfying $a_{j+1}=0$ whenever  $X_{j,j}=X_{j+1,j+1}$, and
        \item $r$ is a real number.
    \end{itemize}

    Moreover, such a matrix $A$ is uniquely determined by the germ at $p$ of the Bochner--K\"{a}hler structure being represented. 
\end{thmA}

Ultimately, this result follows from an explicit calculation in coordinates whose usage is standard within CR geometry (e.g., \cite{MR145555}*{Section 3})
after applying the hard-earned correspondence between Sasakian structures on flat CR manifolds and Bochner--K\"{a}hler structures observed in \cite{MR0520599}. The paucity of known Bochner--K\"{a}hler structure examples noted in \cite{MR1824987} persists today, and an immediate application of Theorem \ref{eqn: potential intro} is that it produces new examples with explicit formulas. We highlight some of these in Section \ref{sec: Examples}.

The moduli space of Bochner--K\"{a}hler structures can be deduced from Theorem \ref{thm: general potentials}. The first computation of this moduli space was derived in \cite{MR1824987}, where it is described as the orbit space of $\I \mathfrak{u}(n)\oplus \mathbb{C}^n\oplus \mathbb{R}$ under a natural action of the unitary group $\operatorname{U}(n)$. To align these descriptions, note that the first part of Theorem \ref{thm: general potentials} remains true if we relax constraints on $X$ and $a$ in \eqref{eqn: potential parameters intro}, instead requiring only that $X$ be Hermitian and $a\in\mathbb{C}^n$. This more general set of matrices is isomorphic to $\I \mathfrak{u}(n)\oplus \mathbb{C}^n\oplus \mathbb{R}$ and is invariant under conjugation by matrices of the form 
\[
\begin{bmatrix}
    1& 0 & 0\\0& U & 0\\ 0&0&1
\end{bmatrix}
\quad\quad\mbox{ with } U\in \operatorname{U}(n),
\]
which defines the relevant action of $\operatorname{U}(n)$ on $\I \mathfrak{u}(n)\oplus \mathbb{C}^n\oplus \mathbb{R}$. Unique representatives for each orbit under this action are given by matrices of the form \eqref{eqn: potential parameters intro} with $X$ and $a$ satisfying the additional constraints of Theorem \ref{thm: general potentials}.

With the moduli space parameterisation of Theorem \ref{thm: general potentials} in hand, we are able to compute the infinitesimal automorphism algebras of all Bochner--K\"{a}hler structures.
\newtheorem*{thmB}{\bf Theorem \ref{thm: infinitesimal symmetries}}
\begin{thmB}
Let $M$ be the (local) Bochner--K\"{a}hler manifold associated with the parameters $a$, $X$, and $r$ of Theorem \ref{thm: general potentials}. The infinitesimal automorphisms of $M$ correspond to the vector fields
\[
(\beta+\I \hBlock z+\alpha w+2\I \langle z,\alpha\rangle z+\rho zw)\frac{\partial}{\partial z}+(2\I \langle z,\beta \rangle + 2\I \langle z,\alpha\rangle w+\rho w^2)\frac{\partial}{\partial w},
\]
where 
\[
    \alpha=\I X \beta 
    \quad\mbox{ and }\quad\rho =2\Im \langle \beta,a\rangle
\]
and the Hermitian matrix $\hBlock$ and $\beta\in\mathbb{C}^n$ satisfy 
\[
    [\hBlock,X]=-2\I(a\beta^*+\beta a^*)
    \quad\mbox{ and }\quad
    \I \hBlock a=(X^2+r)\beta.
\]
\end{thmB}
A more detailed description of these symmetry algebras is given in Section \ref{sec: ISA} (Proposition \ref{prop: symmetry algebras}).

The structure of the paper is as follows. In Section 2, we recall the 
fundamental identification between Bochner and Chern-Moser curvature. 
In Section 3, we describe locally Bochner flat K\"{a}hler structures and the corresponding moduli space. 
In Section 4, we prove Theorem 1, providing a complete description of K\"{a}hler potentials defining Bochner--K\"{a}hler structures. 
 Section 5 is devoted to examples, presenting known examples within our new framework, and providing some new closed form examples. 
In Section 6, we prove Theorem \ref{thm: infinitesimal symmetries}, and present a complete classification of the K\"{a}hler symmetry algebras associated with each point in the normalised moduli space parameters of Theorem \ref{thm: general potentials} for Bochner--K\"{a}hler structures on complex surfaces and complex $3$-manifolds.

\section{Metric Curvature and CR curvature}
Consider the CR manifold $N \subseteq \mathbb C^{n+1}$ given in coordinates $w, z_1, \dots z_n$ as a graph 
\begin{equation}\label{eqn: graph form}
v=F(z,u),
\end{equation}
where $w = u+\I v$.

If $N$ is Sasakian locally around $0$ with distinguished Reeb vector field $\reeb=\frac{\partial}{\partial u}$ then (locally) $F$ does not depend on $u$, and $F$ can be regarded as a (local) potential of the K\"{a}hler form $\omega=\tfrac{\I}{2}\partial \overline{\partial} F$ on a complex manifold $M$ in local coordinates $(z_1,\ldots, z_n)$. This is the coordinate realisation of a correspondence between $(2n+1)$-dimensional Sasakian structures and $2n$-dimensional K\"{a}hler structures. Alternatively, independent of coordinates, the corresponding $(M,\omega)$ can be described locally as the leaf space of the foliation generated by the Sasakian structure's distinguished Reeb field with K\"{a}hler structure induced from the CR structure's Levi form and almost complex structure operator, which indeed descend coherently to the leaf space as it is formed via factorising by the action of a one-parametric CR symmetry subgroup. In \cite{MR0520599}, the Chern--Moser curvature tensor of the CR structure on $M$ is identified with the Bochner curvature tensor on the corresponding K\"{a}hler structure of $(M,\omega)$, and we recall now fundamentals of this identification, as they are essential for the sequel.

The Chern--Moser curvature tensor at a given reference point can be read from the Chern--Moser normal form of the defining equation. Recall that a real-analytic defining equation is in Chern--Moser normal form if
$$F=\langle z,z\rangle + F_{22}+ F_{32}+F_{23}+F_{33} +\sum_{\stackrel{\min(j,k)\ge 2}{\max(j,k)\ge4}} F_{jk},$$
where $F_{jk}$ are polynomials in $z$ of degree $j$ and in $\bar{z}$ of degree $k$ with coefficients being analytic functions of $u$, and the terms $F_{22}$, $F_{32}=\overline{F_{23}}$ and $F_{33}$ satisfy the trace conditions $\tr F_{22}=0$, $\tr^2 F_{32}=0$, and $\tr^3 F_{33}=0$, \cite{MR425155}.

The Chern--Moser invariant at the origin is the tensor $S(0)=F_{22}|_{u=0}$. If $S = 0$ in an open neighbourhood then $N$ is locally equivalent to the Heisenberg sphere
$$v=\langle z,z\rangle$$
as a CR manifold. The Chern--Moser normalisation process does not respect the infinitesimal automorphism $\reeb$, so for Sasakian manifolds an alternative ``rigid'' normal form is needed, that is, a normal form from which the distinguished infinitesimal automorphism can be recovered. Here ``rigid'' refers to the terminology of ``rigid hypersurfaces''\cite{baouendi1985cr} (also referred to in \cite{MR145555} as ``regular''), with defining equations that do not depend on the variable $u$. 
The rigid normalisation process is nothing but the first steps of the Chern--Moser normalisation applied to \eqref{eqn: graph form}, ignoring the trace conditions. In the resulting equation the terms $F_{jk}$ are constants, i.e., they do not depend on $u$.

The Riemann curvature of $M$ is given by   
$$R_{k,\bar{j},m}^n=-G_{m\bar{s},k\bar{j}}(G^{-1})^{\bar{s}n}+ G_{m\bar{s},\bar{j}}(G^{-1})^{\bar{s}r}G_{r\bar{t},k}(G^{-1})^{\bar{t}n},$$
where $G$ is the metric tensor with $G_{m\bar{n}}=\frac{\partial^2F}{\partial z_m\partial z_{\bar{n}}}$, and $G_{m\bar{s},k\bar{j}}$ and $G_{m\bar{s},\bar{j}}$ denote $\frac{\partial^4F}{\partial z_{\bar{j}} \partial z_k \partial z_{\bar{s}}\partial z_{m}}$ and $\frac{\partial^3F}{\partial z_{\bar{j}} \partial z_{\bar{s}}\partial  z_m}$, respectively. See, e.g., \cite{MR4729636}.  If $F$ is in normal form, we have $G(0)=\id$, $G_{,k}(0)=0$, with $\id$ denoting the identity matrix. Hence,
\[
R(0)=-G_{,k\bar{j}}(0)=-F_{22}.
\]
 
 Assume that $N$ is rigid and locally given in \emph{rigid normal form} (see \cites{St91,MR3833790}).  Recall that the rigid normal form can be achieved in a simple two-step procedure starting from \eqref{eqn: graph form}: First, the pluriharmonic terms $F_{k0}$ and $F_{0k}=\overline{F_{k0}}$ ($k\ge 0$) are removed by the transformation $w-2\I \sum F_{k0} \to w$. Second, the terms $F_{k1}$  ($k> 1$), which have the form $F_{k1}=\sum_{j=1}^n \alpha_k^j \bar{z}_j$ for some coefficients $\alpha_k^{j}$, along with  $F_{1k}=\overline{F_{k1}}$, are removed by the transformation $\sum_{k=1}^\infty \alpha^j_{k}\to z_j$. 
  Then the next step in order to achieve Chern--Moser normal form is a coordinate transformation
 \begin{align*}
 z&=C(W)Z\\
 w&=h(W)
 \end{align*}
 where $C|_{w=u}$ is unitary with respect to the hermitian form $\langle\cdot,\cdot\rangle$. This results in the elimination of the trace of $F_{22}.$ The resulting traceless $F_{22}$ in normal form is exactly the Chern--Moser invariant $S$, if $n>1$.  It follows that $M$ has vanishing Bochner curvature at the reference point if and only if $N$ is CR umbilic (i.e., vanishing $S$) at the reference point.  It is well known that CR umbilicity in an open neighbourhood implies local CR flatness, that is, local CR equivalence to the Heisenberg sphere. This proves

\begin{prop}\label{BK manifold to sphere correspondence}
The K\"{a}hler manifold $M$ of dimension $n\ge 2$ and of class $\mathcal C^5$ is Bochner flat if and only if the corresponding Sasakian manifold $N$ is spherical, that is, locally CR equivalent to the Heisenberg sphere
$$v=\sum_{j=1}^n |z_j|^2.$$ 
\end{prop}

{\bf Remark.} A K\"{a}hler manifold with a metric tensor of class $\mathcal C^5$ has a K\"{a}hler potential of class $\mathcal C^7$, which is sufficient for the complete Cartan prolongation of the corresponding CR manifold, and hence to conclude its CR flatness.  According to a recent result by Kossovskiy and Zaitsev \cite{KZ}, a strictly pseudoconvex hypersurface $M$ in $\mathbb C^2$ is spherical if it is of class $\mathcal C^6$ and the fundamental Chern--Moser invariant vanishes identically. This result is believed to remain true in higher dimensions for rigid, strictly pseudoconvex CR hypersurfaces of class $\mathcal C^4$, which applies to our situation and reduces the required \emph{a priori} smoothness of the metric tensor to $\mathcal C^2$. Notice that the correspondence with spherical Sasakian manifolds implies \emph{a posteriori} real analitycity.

\section{Bochner--K\"{a}hler moduli and CR symmetries}\label{sec: moduli and symmetries}

Locally, a spherical Sasakian manifold is completely determined by the choice of an infinitesimal CR automorphism on a spherical CR manifold that is also a Reeb vector field. The Lie algebra of infinitesimal automorphisms of the Heisenberg sphere is the $(n^2+4n+3)$-dimensional Lie algebra $\mathfrak{su}(n+1,1)$, which has a grading
\[
\mathfrak{su}(n+1,1)=\mathfrak g_{-2}\oplus \mathfrak g_{-1}\oplus \mathfrak g_{0}\oplus \mathfrak g_{1}\oplus \mathfrak g_{2},
\]
where $\mathfrak{g}_{0}\oplus \mathfrak g_{1}\oplus \mathfrak g_{2}$ corresponds to an isotropy subalgebra composed of all infinitesimal automorphisms that vanish at some fixed point $p\in N$.

Any infinitesimal automorphism that is transverse to the CR distribution (i.e., it has a non-zero $\mathfrak g_{-2}$ component) can be the Reeb vector field of a Sasakian manifold. Different infinitesimal automorphisms will give rise to equivalent Sasakian manifolds if they are related by a CR automorphism of the Heisenberg sphere. 

Applying Proposition \ref{BK manifold to sphere correspondence}, we will locally describe Bochner flat K\"{a}hler structures as a quotient manifold of the germ at $(w,z)=0$ of the Heisenberg sphere
\begin{align}\label{Heisenberg sphere}
\Im w =\langle z,z\rangle= \sum_{j=1}^n |z_j|^2
\end{align}
equipped with a distinguished transverse infinitesimal automorphism, which has the form
\begin{align}\label{Reeb VF}
\reeb=&(\hat{b}+(\hat{s}\id +\I \hat{X})z+\hat{a}w+2\I \langle z, \hat{a} \rangle z + \hat{r} zw) \frac{\partial}{\partial z} + \\ 
& + (1+ 2\I \langle z,\hat{b}\rangle + 2\hat{s}w+ 2\I \langle z,\hat{a}\rangle w+ \hat{r}w^2)\frac{\partial}{\partial w},\nonumber
\end{align}
where $\hat{b},\hat{a}\in \mathbb C^n$, $\hat{s},\hat{r}\in\mathbb R$ and $\hat{X}$ is a Hermitian matrix. Up to scaling, this is the general form for $\Re(\reeb)$ to be a Reeb vector field in an open neighbourhood of $0$, and the local quotient manifold modulo the flow of $\Re(\reeb)$ is a Bochner--K\"{a}hler manifold. 

Our next task is to normalise the parameters $\hat{b},\hat{a}, \hat{s},\hat{r},\hat{X}$ in \eqref{Reeb VF}, achieving $\hat{b}=0$ and $\hat{s}=0$ in particular, and throughout the sequel we use undecorated labels ${a},{r},{X}$ for the parameters constrained to their normal forms. In Section \ref{sec: potentials thm proof}, we discuss an alternate partial normalisation with $\hat{b}\neq 0$, and the undecorated $b$ will be used only then for this alternate normalisation. 

 CR automorphisms of the Heisenberg sphere  also act on the Lie algebra of infinitesimal automorphisms by the adjoint action. They  result in equivalences of the resulting K\"{a}hler manifolds if they change $\reeb$, and they  result in automorphisms (isometries) of the K\"{a}hler manifold if they preserve $\reeb$. Similarly, to describe equivalence of germs of Bochner--K\"{a}hler manifolds at $0$, as we will, one should instead consider the smaller orbits of the action by the isotropy subgroup inside the CR automorphisms.

The Lie algebra, $\mathfrak{su}(n+1,1)$, of infinitesimal automorphisms of the Heisenberg sphere can be represented by

\begin{equation}\label{indef su rep}
\begin{bmatrix}-\hat{s}-\I c & -2\I \hat{a}^* & \hat{r} \\ \hat{b} & \I \hat{X} -\I c &  \hat{a}\\ \hat{q} & 2\I \hat{b}^* & \hat{s} -\I c \end{bmatrix}
,
\end{equation}
which are skew-hermitian with respect to
\begin{equation}\label{eqn: indef su rep h-form}
\begin{bmatrix} 0 & 0 &\frac{\I}{2} \\ 0 & \id & 0\\ -\frac{\I}{2} & 0 & 0 \end{bmatrix}.
\end{equation}
 Here $\hat{q}\in\mathbb R$, $\hat{b},\hat{a} \in\mathbb C^n$, $\hat{X}$ is hermitian, and $\hat{s},\hat{r}\in\mathbb R$ are as in \eqref{Reeb VF}, whereas $c$ parameterises different elements in a common coset, regarding $\mathfrak{su}(n+1,1)$ as the coset space $\mathfrak{cu}(n+1,1)/\mathbb{C}=(\mathfrak{u}(n+1,1)\oplus \mathbb{R})/\mathbb{C}$.

This representation comes from a standard construction in CR geometry (e.g., \cite{MR145555}*{Section 4}) whereby the sphere $S^{2n+1}\subset \mathbb{C}^{n+1}$ is regarded as the projectivised null cone of a Lorentzian-signature Hermitian form on $\mathbb{C}^{n+2}$. Specifically, taking $C:=\{v\in \mathbb{C}^{n+2}\,|\, \langle v,v\rangle=0\}$, where $\langle\cdot,\cdot\rangle$ is the Hermitian form represented by \eqref{eqn: indef su rep h-form} in some coordinates
\[
v=(\psi,\phi,\omega)\in\mathbb{C}\oplus\mathbb{C}^n\oplus\mathbb{C}\cong\mathbb{C}^{n+2},
\]
we have $\lambda v\in C$ for all $\lambda\in\mathbb{C}$ and $v\in C$, so $C$ can be projectivised to define the real hypersurface $\mathbb{P}(C)$ in $\mathbb{CP}^{n+1}$. The hypersurface $\mathbb{P}(C)$ is indeed CR spherical, as it locally has the forms 
\[
\big\{\Im(w)=\sum z_{j}\overline{z_j}\big\}=\mathbb{P}(C)\cap\left\{[\psi:\phi:\omega]\in \mathbb{CP}^{n+1}\,|\,\psi\neq 0 \right\}
\]
and
\[
\big\{\Im(w^\prime)=\sum z_{j}^\prime\overline{z_j^\prime}\big\}=\mathbb{P}(C)\cap\left\{[\psi:\phi:\omega]\in \mathbb{CP}^{n+1}\,|\,\omega\neq 0 \right\},
\]
on open sets congruent to $\mathbb{C}^{n+1}$, where $(z,w)=\left(\frac{\phi}{\psi},\frac{\omega}{\psi}\right)$ and $(z^\prime,w^\prime)=\left(\frac{\phi}{\omega},\frac{\psi}{\omega}\right)$. The group $U(n+1,1)$ of biholomorphisms on $\mathbb{C}^{n+2}$ preserving the Hermitian form \eqref{eqn: indef su rep h-form} preserves $\mathbb{P}(C)$ under its induced action. The special unitary subgroup $\operatorname{SU}(n+1,1)$ generates the same induced symmetries, and it is known to produce all of the CR symmetries of $\mathbb{P}(C)$, which identifies the sphere's infinitesimal symmetries with \eqref{indef su rep}.

 In order to make the matrix \eqref{indef su rep} traceless we need to adjust the trace by choosing $c=\frac{1}{n+2} \operatorname{tr} \hat{X}$, yielding the standard representation. Below, we mostly use the representation of 
 $\mathfrak{su}(n+1,1)$ by the quadratic vector fields  \eqref{Reeb VF} for which the trace is adjusted by the choice $c=\I s$. For such choice $\mathfrak{su}(n+1,1)$ should be interpreted as the quotient  $\mathfrak{cu}(n+1,1)/\mathbb{C}$ and one should keep in mind that the adjustment of the trace does not affect the following calculations with commutators. We are interested in infinitesimal automorphisms with $\hat{q}\neq 0$. After scaling by a constant we may assume that $\hat{q}=1$. This just scales the resulting K\"{a}hler metric by a constant.

The moduli space of the (local) Bochner--K\"{a}hler manifolds is now the quotient space of $\mathfrak{su}(n+1,1)$ (with $q\neq 0$) modulo the adjoint action of the CR isotropy subgroup inside of $\operatorname{SU}(n+1,1)$.
See Theorem 3.1 in \cite{MR1824987}. For the germs at $0$ we may look at the quotient space with respect to the subgroup that preserves the origin. In the representation \eqref{indef su rep} with $c=\I \hat{s}$, this subgroup consists of the triangular matrices of the form
\begin{equation}\label{eqn: CR isotropy representation}
\begin{bmatrix}1 & -2\I \alpha^* &-\rho-\I \langle \alpha,\alpha\rangle \\ 0 & tU &  \alpha\\ 0 & 0 & t^2  \end{bmatrix},
\end{equation}
(which can be re-scaled to determinant $1$ by instead using the representation with a different choice for $c$, but that doesn't affect the following calculations).  

Consider the adjoint action by this subgroup 
\begin{multline*}
\begin{bmatrix}1 & 2\I \alpha^* &\rho-\I \langle \alpha,\alpha\rangle \\ 0 & \id  &  -\alpha\\ 0 & 0 & 1  \end{bmatrix}\begin{bmatrix}-\hat{s} & -2\I \hat{a}^* &\hat{r} \\ \hat{b} & \I \hat{X}  &  \hat{a}\\ 1 & 2\I \hat{b}^* & \hat{s}  \end{bmatrix}\begin{bmatrix}1 & -2\I \alpha^* &-\rho-\I \langle \alpha,\alpha\rangle \\ 0 & \id &  \alpha\\ 0 & 0 & 1  \end{bmatrix}\\= \begin{bmatrix}*& * & *\\ \hat{b}-\alpha &* &  *\\ 1 & 2\I({\hat{b}}^*-\alpha^*) & * \end{bmatrix}.
\end{multline*}
By choosing $\alpha=\hat{b}$ we can eliminate the parameter $\hat{b}$ in the vector field. The action of the subgroup
$$
\begin{bmatrix}1 & 0 & \frac{\rho}{t^2} \\ 0 & \frac1{t} U^{-1}  & 0\\ 0 & 0 & \frac{1}{t^2}  \end{bmatrix}\begin{bmatrix}-\hat{s} & -2\I \hat{a}^* &\hat{r} \\ 0 & \I \hat{X}  &  \hat{a}\\ 1 & 0 & \hat{s}  \end{bmatrix}\begin{bmatrix}1 & 0 &-\rho \\ 0 & tU &  0\\ 0 & 0 & t^2  \end{bmatrix}= \begin{bmatrix}-\hat{s}+\frac{\rho}{t^2}& -2\I t \hat{a}^*U& *\\ 0 &\I U^{-1}\hat{X}U &  tU^{-1} \hat{a}\\ \frac{1}{t^2} & 0 & \hat{s}-\frac{\rho}{t^2} \end{bmatrix},
$$
where $U$ is a unitary matrix, $t$ a nonzero real parameter and $\rho$ a real parameter, preserves $\hat{b}=0$. We choose $t=1$ to keep $\hat{q}=1$ and $\rho=\hat{s}$ to eliminate $\hat{s}$. Going forward, we will use the undecorated labels $X$, $a$, and $r$ to represent parameters determining $\reeb$ in this normalised form, which is
\begin{equation}\label{Reeb VF normalised}
\reeb=\left(\I X z+aw+2\I \langle z, a \rangle z + r zw\right) \frac{\partial}{\partial z} + \left(1 + 2\I \langle z,a\rangle w+ rw^2\right)\frac{\partial}{\partial w}.
\end{equation}

This leaves the parameters $\hat{X}$, $\hat{a}$, $\hat{r}$ modulo the adjoint action of $U^{-1}$ on $\hat{X}$, the standard action of $U^{-1}$ on $\hat{a}$ and the trivial action on $\hat{r}$. This remaining action by $\operatorname{U}(n)$ can furthermore be applied to diagonalise $X$ and normalise $a$ so that it has all nonnegative real-valued entries with a maximal number of zeros. Up to a permutation of basis elements, this fully normalises the parameters $a$, $X$, and $r$ representing $\reeb$. Altogether we have established the following.

\begin{prop}\label{prop2}
The moduli space of (local) Bochner--K\"{a}hler manifolds is the quotient space of $\mathcal H_n\oplus \mathbb C^n \oplus\mathbb R$ modulo the $\operatorname{U}(n)$ action described above, with $(X,a,r)$ regarded as general elements in $\mathcal H_n\oplus \mathbb C^n \oplus\mathbb R$. Here $\mathcal H_n$ is the space of hermitian $n\times n$ matrices. By means of the $\operatorname{U}(n)$ action one can normalise $(X,a,r)$ such that
    \begin{itemize}
        \item $X$ is diagonal with entries ordered such that $X_{j,j}\geq X_{j+1,j+1}$, and
        \item $a_{j+1}=0$ whenever  $X_{j,j}=X_{j+1,j+1}$,
    \end{itemize}
which identifies the moduli space with a subset in $\mathbb R^n \times \mathbb R_{\ge 0}^n\times \mathbb R$ containing a dense open subset. 

If the eigenvalues of $X$ are distinct then the subgroup of $\operatorname{U}(n)$ that keeps $X$ diagonal is $\operatorname{U}(1)^n$, which can be used to make all entries of $a$ real and non-negative. Such cases cover an open and dense piece of the moduli space. 
\end{prop}
For example, if $X$ is a multiple of the identity then the $\operatorname{U}(n)$ action can be used to make $a=[a_1,0,\dots,0]$, where $a_1$ is real and non-negative.

\begin{lem}
    Let $(N,\reeb)$ be a spherical Sasakian structure with $\reeb$ given by parameters $X$, $a$, and $r$ that are normalised as in Proposition \ref{prop2}. The isotropy subgroup at $0$ of the induced Bochner--K\"{a}hler structure is 
    isomorphic to
    \[
    \{U\in \operatorname{U}(n)\,|\, U^{-1}XU=X,\,U^{-1}a=a\}.
    \]
\end{lem}
\begin{proof}
    This matrix group exactly parameterises the subgroup of the CR isotropy group \eqref{eqn: CR isotropy representation} preserving $\reeb$.
\end{proof}

It was shown in \cite{MR3833790} that the parameters $X$, $a$, and $r$ in \eqref{Reeb VF normalised} are related to the trace terms of the rigid normal form as follows:

\begin{align*}
F_{22}&= -2 \langle X z,z\rangle\langle z,z\rangle\\
\tr F_{22}&= -2(\tr X \langle z,z\rangle+(n+2)\langle X z,z\rangle)\\
\tr^2 F_{22}&= -4(n+1) \tr X\\
F_{23}&= -2 \langle a,z \rangle \langle z,z\rangle^2\\
\tr^2 F_{23}&= -4 (n+2)(n+1) \langle a,z\rangle\\
F_{33}&= -\frac{2}{3}r \langle z,z\rangle^3 + 4 \langle X z,z\rangle^2\langle z,z\rangle+ 2\langle X^2 z,z\rangle \langle z,z\rangle^2 \\
\tr^3F_{33}&= -4(n+2)(n+1)n r +12 (n+2)(n+3)\tr X^2  + 24(n+2) (\tr X)^2,
\end{align*}
where $\tr$ denotes the \emph{trace} differential operator defined with respect to the standard hermitian product, i.e.,  $\tr := \sum_{j=1}^n \frac{\partial^2}{\partial z_j\partial \bar{z}_j}$.

\section{Complete Description of Bochner--K\"{a}hler potentials}\label{sec: potentials thm proof}

In this section, we prove the first main theorem.
\begin{theorem}\label{thm: general potentials}
    A real-valued function $F$ on a complex manifold $M$ is a K\"{a}hler potential of a Bochner--K\"{a}hler structure at a point $p\in M$ if and only if
    \begin{equation}\label{eqn: potential}
        \begin{bmatrix} 0& z^* & \frac{\I}{2} \end{bmatrix} \e^{2\I AF}\begin{bmatrix} 1\\ z\\ 0 \end{bmatrix}=0
    \end{equation}
    in some coordinates $z=(z_1,\ldots,z_n)$ centered at $p$, for some matrix $A\in \mathfrak{gl}_{n+2}(\mathbb{C})$ of the form 
    \begin{equation}\label{eqn: potential parameters}   
    A=\begin{bmatrix}
        0& -2\I a^* & r\\0&\I X& a\\1&0&0
    \end{bmatrix},
    \end{equation}
    where 
    \begin{itemize}
        \item $X\in \mathfrak{gl}_{n}(\mathbb{R})$ is diagonal with entries ordered such that $X_{j,j}\geq X_{j+1,j+1}$,
        \item $a$ is a vector with non-negative real entries satisfying $a_{j+1}=0$ whenever  $X_{j,j}=X_{j+1,j+1}$, and
        \item $r$ is a real number.
    \end{itemize}

    Moreover, such a matrix $A$ is uniquely determined by the germ at $p$ of the Bochner--K\"{a}hler structure being represented. 
\end{theorem}

Before proving Theorem \ref{thm: general potentials}, we will review special cases of this result that are already known, as Theorem \ref{thm: general potentials} is then obtained with similar arguments. The special case we will consider is where the Sasakian structure is given by \eqref{Reeb VF} with parameter constraints $\hat{b}=b$, $\hat{s}=s$, $\hat{a}=0$, $\hat{X}=\theta$, and $\hat{r}=0$, that is,
\begin{equation}\label{Reeb VF lower triangular}
\reeb=\left(b+(s\id +\I \theta)z\right) \frac{\partial}{\partial z} + \left(1+ 2\I \langle z,b\rangle + 2sw\right)\frac{\partial}{\partial w}.
\end{equation}
Or, to express $\reeb$ in the matrix representation of \eqref{indef su rep} (with convention $c=\I \hat{s}$), in homogeneous coordinates $z=\frac{\phi}{\psi}$ and $w=\frac{\omega}{\psi}$ this becomes 
\[
\reeb = 
\begin{bmatrix} \frac{\partial}{\partial \psi}& \frac{\partial}{\partial \phi} & \frac{\partial}{\partial \omega} \end{bmatrix}\begin{bmatrix} 0& 0 & 0\\ b & s+\I \theta & 0 \\ 1& 2\I b^* &2s \end{bmatrix}\begin{bmatrix} \psi\\\phi\\\omega 
\end{bmatrix}
\]
with $b\in\mathbb C^n$, $s\in\mathbb R$, and $\theta$ Hermitian.
\begin{remark}\label{rem: parameter conventions}
    The special case \eqref{Reeb VF lower triangular} is not in the normal form \eqref{Reeb VF normalised}. Applying Section \ref{sec: moduli and symmetries} normalisations transforms \eqref{Reeb VF lower triangular} to the normal form \eqref{Reeb VF normalised} with 
    \begin{equation}\label{eqn: stanton to general normal coordinates}
    X= \theta-2 bb^*-\langle b,b\rangle \id,
    \quad
    a= (\I \theta -s\id-2\I \langle b,b\rangle\id)b,
    \quad
    r=s^2 -2b^*\theta b +3 \langle b,b\rangle^2.
    \end{equation} 
    
    Rigid CR structures in the form \eqref{Reeb VF lower triangular} are analysed in \cites{MR3833790,St91}, and the parameters in \eqref{Reeb VF lower triangular} directly match the parameters used there, except that the label $r$ is used instead of $s$ in both articles. 
\end{remark} 

Explicit expressions of the potentials of the Bochner--K\"{a}hler manifolds, or equally, the defining equations of rigid spheres can be found by applying normalisations to a Heisenberg sphere that map a given the Sasakian structure defining symmetry $\reeb$ of the Heisenberg sphere to $\frac{\partial}{\partial u}$.

Stanton \cite{St91} produced the corresponding (inverse) mappings $Z=Z(z,w)$, $W=W(z,w)$ for the special case of \eqref{Reeb VF lower triangular} with $n=1$. Ezhov, Kol\'a\v{r} and Schmalz \cite{MR3833790} generalised this to arbitrary dimensions. The mappings satisfy the system of linear ODE
\begin{align*}
    \frac{\partial Z}{\partial w}&=b+(s+\I \theta)Z\\
    \frac{\partial W}{\partial w}&=1+2\I \langle Z,b\rangle +2s W
\end{align*}
with $b,s,\theta$ as in \eqref{Reeb VF lower triangular},  subject to the initial conditions
\[
Z(z,0)\equiv \frac{z}{1-2\I \langle z,b\rangle}
\quad\mbox{ and }\quad
W(z,0)\equiv 0.
\]
These initial conditions make sure that the mapping gives a defining equation in rigid normal form,  which means, in particular, that its Taylor series does not contain pluriharmonic terms. The solution is

\begin{align*}
Z(z,w)=& R^{-1}(\e^{Rw}-\id)b + \frac{\e^{Rw}z}{1-2\I\langle z,b\rangle}\\
W(z,w)=& (1-2\I \langle R^{-1}b, b\rangle)\frac{\e^{2sw}-1}{2s} + 2\I \e^{2sw}\langle \tilde{R}^{-1}(\id-\e^{-\tilde{R}w})(\frac{z}{1-2\I\langle z,b \rangle}+R^{-1}b),b \rangle\\
=& \frac{\e^{2sw}-1}{2s} -2\I \left\langle \left[\frac{\e^{2sw}-1}{2s}- (2s-R)^{-1}(\e^{2sw}-\e^{Rw}) \right] R^{-1}b,b\right\rangle \\ &+ \frac{2\I\e^{2sw}}{1-2\I\langle z,b \rangle} \langle \tilde{R}^{-1}(\id-\e^{-\tilde{R}w})z,b \rangle,
\end{align*}
for matrices $R=s \id+\I\theta$ and $\tilde{R}=s\id -\I\theta$, and the resulting (implicit) equations
\begin{multline}\label{eqn: defining eqn in Stanton coordinates}
\frac{\sin 2sv}{2s}(1-2\langle \theta b,R^{-1}\tilde{R}^{-1} b\rangle)= \frac{\langle \e^{-2\theta v}z,z\rangle}{|1-2\I \langle z,b \rangle|^2} + \langle(\e^{-2\theta v}-\cos 2sv) b , R^{-1}\tilde{R}^{-1}b \rangle\\+ \frac{\langle  (\e^{-2 \theta v}-\e^{2\I sv})z, R^{-1}b\rangle}{1-2\I \langle z,b \rangle}+
\frac{\langle R^{-1}b, (\e^{-2 \theta v}-\e^{2\I sv})z \rangle}{1+2\I \langle b,z \rangle}
\end{multline}
realise an open subset of all rigid spheres. We derive \eqref{eqn: defining eqn in Stanton coordinates} in Section \ref{sec: Examples}.

Notice that the reduction of the vector field $\mathcal Z$ to the form \eqref{Reeb VF normalised} is always possible, while the reduction to the form \eqref{Reeb VF lower triangular} is not. If the latter reduction can be achieved it is preferable because it produces a simpler triangular structure. Section \ref{Stanton coordinates and the a=0 structures.} presents a structure that cannot be reduced to \eqref{Reeb VF lower triangular}.

We approach the general case of Theorem \ref{thm: general potentials} with a similar analysis, but our starting point is the general normal form of \eqref{Reeb VF normalised} rather than the special case of \eqref{Reeb VF lower triangular}.

\begin{proof}[Proof of Theorem \ref{thm: general potentials}]
We would like to find a change of coordinates on the Heisenberg sphere that transforms $\mathcal{Z}$ given by parameters $a$, $X$, and $r$ in \eqref{Reeb VF normalised} into $\frac{\partial}{\partial u}$. In contrast to the former special case \eqref{Reeb VF lower triangular}, the corresponding system of ODE's is not linear any more if considered in affine coordinates, but it is linear in projective coordinates: 
$$\frac{\partial}{\partial w} \begin{bmatrix} \psi\\\zeta\\\omega \end{bmatrix} = \begin{bmatrix} 0& -2\I a^* & r\\0& \I X & a\\ 1& 0& 0
\end{bmatrix} \begin{bmatrix} \psi\\\zeta\\\omega \end{bmatrix},$$
subject to the initial conditions $\psi(z,0)=1$, $\zeta(z,0)=z$, $\omega(z,0)=0$.
It follows
$$ \begin{bmatrix} \psi\\\zeta\\\omega \end{bmatrix}= \e^{Aw} \begin{bmatrix} 1\\z \\0 \end{bmatrix}$$
where
$$A=\begin{bmatrix} 0& -2\I a^* & r\\ 0 & \I X & a \\ 1& 0 &0 \end{bmatrix},$$
resulting in the rigid sphere equation
\[
\begin{bmatrix} 1& z^* &0\end{bmatrix} \e^{A^*\bar{w}}J\e^{Aw}\begin{bmatrix} 1\\z\\ 0 \end{bmatrix}=
\begin{bmatrix} \psi^*& \zeta^* &\omega^* \end{bmatrix}\begin{bmatrix} 0&0&\frac{\I}{2}\\0& \id & 0\\-\frac{\I}{2}&0&0 \end{bmatrix}\begin{bmatrix} \psi\\ \zeta\\ \omega \end{bmatrix}=0,
\]
where 
$$J=\begin{bmatrix} 0&0&\frac{\I}{2}\\0& \id & 0\\-\frac{\I}{2}&0&0 \end{bmatrix}.$$
Since $A^*J=-JA$, we have
$$\e^{A^*\bar{w}}J= \sum_{n=0}^\infty \frac{(A^*)^nJ \bar{w}^n}{n!}=\sum_{n=0}^\infty \frac{JA^n (-\bar{w})^n}{n!}=J\e^{-A\bar{w}}.$$

Applying the Baker--Campbell--Hausdorff formula, we get the implicit formula
\[
\begin{bmatrix} 0& z^* & \frac{\I}{2} \end{bmatrix} \e^{2\I Av}\begin{bmatrix} 1\\ z\\ 0 \end{bmatrix}=0
\]
for a general Bochner--K\"{a}hler potential $v$. Combining this with Proposition \ref{prop2} completes the proof of Theorem \ref{thm: general potentials}.
\end{proof}

\section{Examples, Old and New}\label{sec: Examples}

In this section we present several examples of Bochner--K\"{a}hler structures, distinguished by constraints on the moduli space parameters $(a,X,r)$ from Theorem \ref{thm: general potentials}. We also express some examples in the parameters $(b,\theta, s)$ with the \emph{generalised Stanton formula}, which are related to $(a,X,r)$ as described in Remark \ref{rem: parameter conventions}. Many of these examples are well known structures, and our aim is to give their realisations in the $(a,X,r)$ space. Section \ref{sec: new closed form examples} presents a family of completely new Bochner--K\"{a}hler structures with explicit formulas for their K\"{a}hler potentials.

\subsection{Homogeneous structures and the general \texorpdfstring{$a=0$}{a=0} formula}\label{sec: a is 0 case}
Consider the Bochner--K\"{a}hler structures given by the defining equation in Theorem \ref{thm: general potentials} with $a=0$. This class contains most of the Bochner--K\"{a}hler structures that have been explicitly described in the current literature. In particular, it contains all homogeneous Bochner--K\"{a}hler structures (see Corollary \ref{cor: homogeneous structures}), the weighted projective spaces studied in \cite{MR1824987}, and the \emph{rotationally invariant} structures studied in \cite{MR0266121}.

With $a=0$, the defining equation \eqref{eqn: potential} simplifies significantly. Indeed, if
\[
A=\begin{bmatrix} 0& 0 & r\\ 0 & \I X & 0 \\ 1& 0 &0 \end{bmatrix} 
\]
then the matrix exponential simplifies to 
\[
\e^{2\I AF}=
\begin{bmatrix} \mu_{1,1}& 0 & \mu_{1,2}\\ 0 & \e^{-2 FX} & 0 \\ \mu_{2,1}& 0 & \mu_{2,2} \end{bmatrix} 
\quad\mbox{ where }\quad
\operatorname{exp}\left({2\I F \begin{bmatrix} 0& r\\ 1 & 0 \end{bmatrix} }\right)=:
\begin{bmatrix} \mu_{1,1}& \mu_{1,2}\\ \mu_{2,1} & \mu_{2,2} \end{bmatrix},
\]
and since $\left(2\I F\begin{bmatrix} 0& r\\ 1 & 0 \end{bmatrix}\right)^2=-4F^2 r\id$, we have
\[
\operatorname{exp}\left(2\I F\begin{bmatrix} 0& r\\ 1 & 0 \end{bmatrix} \right)=\left(\sum_{i=0}^\infty \frac{(2 F \sqrt{-r})^{2i}}{(2 i)!}\right)\id + \I\left(\sum_{i=0}^\infty \frac{(2 F \sqrt{-r})^{2i+1}}{(2 i+1)!}\right) \begin{bmatrix} 0& -\sqrt{-r}\\ \frac{1}{\sqrt{-r}} & 0 \end{bmatrix}
\]
for all $r\neq0$.
Therefore,
\[
\e^{2\I AF} = 
\begin{bmatrix} \cosh(2F\sqrt{-r})& 0 & -\I\sqrt{-r}\sinh(2F\sqrt{-r})\\ 0 & \e^{-2 FX} & 0 \\ \frac{\I\sinh(2F\sqrt{-r})}{\sqrt{-r}}& 0 &\cosh(2F\sqrt{-r}) \end{bmatrix} 
\]
in this special case with $r\neq0$. So, supposing that $X$ is normalised as in Theorem \ref{thm: general potentials}, namely, in the form
\begin{equation}\label{eqn: normalised X}
X=\begin{bmatrix} \tau_1& 0 & 0\\ 0 & \ddots & 0 \\ 0& 0 &\tau_n \end{bmatrix} ,
\end{equation}
the defining equation \eqref{eqn: potential} for the K\"{a}hler potential $F$ becomes
\begin{equation}\label{eqn: examples branch 1a}
\frac{\sinh(2F\sqrt{-r})}{2\sqrt{-r}}=\sum_{i=1}^n\e^{-2 F \tau_i}|z_i|^2
\qquad\forall\, r\neq0.
\end{equation}
The chosen branch of square root does not change these formulas, so \eqref{eqn: examples branch 1a} is well posed despite the ambiguity in $\sqrt{-r}$.

If $r=0$ the computation is simpler and results in 
\begin{equation}\label{eqn: examples branch r=0}
F=\sum_{i=1}^n\e^{-2 F \tau_i}|z_i|^2,
\end{equation}
which could also be found as a limit of \eqref{eqn: examples branch 1a}. The complete Bochner--K\"{a}hler structures described in \cite{MR1824987}*{Theorem 4.27} have the form \eqref{eqn: examples branch r=0}. Also, we can alternatively rearrange \eqref{eqn: examples branch 1a} to
\begin{equation}\label{eqn: examples branch 1b}
\frac{\sin(2F\sqrt{r})}{2\sqrt{r}}=\sum_{i=1}^n\e^{-2 F \tau_i}|z_i|^2.
\end{equation}
The K\"{a}hler potential $F$ is described most elegantly by \eqref{eqn: examples branch 1a}, \eqref{eqn: examples branch r=0}, or \eqref{eqn: examples branch 1b} depending on whether the parameter $r$ is negative, zero, or positive, respectively.

Consider the case \eqref{eqn: examples branch 1a} 
 $$\frac{\sinh \tau_0 F}{\tau_0}=\sum_{j=1}^{n} \e^{\tau_j F}  |z_j|^2,$$
 where $\tau_0=2\sqrt{-r}>0$.  After rescaling of $z$, this implicit equation is equivalent to
 \begin{equation}\label{ie}
  \e^{\rho_0F}+\sum_{j=1}^{n} \e^{\rho_j F}  |z_j|^2=1
  \end{equation}
  where $\rho_j=\tau_j-\tau_0$, for $j>0$, and  $\rho_0=-2\tau_0<0$. 
  The implicit equation \eqref{ie} has a solution in a neighbourhood of $z=0$. This can be shown by the implicit function theorem after replacing $|z_j|^2$ by new variables $u_j$. It follows that the solution depends only on $u_j=|z_j|^2$.
   
We derive an implicit formula for the corresponding metric. Differentiation yields 
 $$\rho_0 \e^{\rho_0 F} F_{\alpha} +  \e^{\rho_\alpha  F}  \overline{z_\alpha} +F_\alpha \sum_{j=1}^{n} \rho_j \e^{\rho_j  F}  |z_j|^2 =0,$$
 where $F_\alpha=\frac{\partial F}{\partial z_{\alpha}}$.
 Let
 $$m_1=\rho_0 \e^{\rho_0 F}  + \sum_{j=1}^{n} \rho_j \e^{\rho_j  F}  |z_j|^2 $$
 Then
 $$F_\alpha=\frac{ - \e^{\rho_\alpha  F}  \overline{z_\alpha}}{m_1 }
 \quad\text{ and }\quad
 F_{\bar{\beta}}=\frac{  -\e^{\rho_\beta  F}  z_{\bar{\beta}}}{m_1 }.$$
 
Differentiating again yields
\begin{multline*}
 \rho_0 \e^{\rho_0 F} F_{\alpha\bar{\beta}} + \rho_0^2 \e^{\rho_0 F} F_{\alpha}F_{\bar{\beta}}+  \e^{\rho_\alpha  F}  \delta_{\alpha\bar{\beta}} +\rho_{\alpha}\e^{\rho_\alpha  F}  \overline{z_\alpha} F_{\bar{\beta}}\\ + F_{\alpha\bar{\beta}} \sum_{j=1}^{n} \rho_j \e^{\rho_j  F}  |z_j|^2 + F_\alpha F_{\bar{\beta}} \sum_{j=1}^n \rho^2_j \e^{\rho_j  F}  |z_j|^2+F_\alpha \rho_\beta \e^{\rho_\beta F} z_{\bar{\beta}}=0
 \end{multline*}
 and hence
$$   m_1 F_{\alpha\bar{\beta}} =   -\e^{\rho_\alpha  F}  \delta_{\alpha\bar{\beta}} -  m_2 F_{\alpha}F_{\bar{\beta}}-\rho_{\alpha}\e^{\rho_\alpha  F}  \overline{z_\alpha} F_{\bar{\beta}} -\rho_\beta \e^{\rho_\beta F} z_{\bar{\beta}}  F_\alpha$$
where
\[
m_2=  \rho_0^2 \e^{\rho_0 F}  +\sum_{j=1}^{n} \rho^2_j \e^{\rho_j  F} |z_j|^2.
\]
It follows
\[
F_{\alpha\bar{\beta}}= \frac{  -\e^{\rho_\alpha  F} \delta_{\alpha\bar{\beta}} }{ m_1} +\frac{((\rho_\alpha+\rho_\beta)m_1-m_2) \e^{(\rho_\alpha+\rho_\beta)  F} \overline{z_\alpha}z_{\bar{\beta}}}{ m_1^3}.
\]

\medskip

The following subsections describe more specialised cases in further detail.

\subsubsection{Homogeneous structures}\label{sec: homogeneous structures}
In Corollary \ref{cor: homogeneous structures} (\S\ref{sec: ISA} below), we show that the Bochner--K\"{a}hler structures expressed in normal coordinates $(a,X,r)$ are homogeneous if and only if 
\begin{equation}\label{eqn: homogeneous parameters}
a=0,
\quad
X = \tau\begin{bmatrix}
    \id_{n_1} & 0 \\ 0 & -\id_{n-n_1}
\end{bmatrix},
\quad\mbox{ and }\quad 
r = -\tau^2
\quad\mbox{ for some }\tau\geq 0,\, n_1\in\mathbb{N}
\end{equation}
where $\id_{n}$ denotes an $n\times n$ identity matrix. 
In this case, \eqref{eqn: examples branch 1a} becomes
\[
\frac{\e^{2F\tau}-\e^{-2F\tau}}{4\tau}=\frac{\sinh(2F\tau)}{2\tau}=\sum_{i=1}^{n_1}\e^{-2 F \tau}|z_i|^2 + \sum_{i=n_1}^{n}\e^{2 F \tau}|z_i|^2
\qquad\forall\, r\neq0,
\]
which rearranges to
\[
\e^{4F\tau}\left(1-4\tau\sum_{i=n_1}^{n}|z_i|^2\right)=1+4\tau\sum_{i=1}^{n_1}|z_i|^2,
\]
or equivalently
\begin{equation}\label{eqn: examples branch homogeneous}
F = \frac{1}{4\tau}\left( \ln\left(1+4\tau\sum_{i=1}^{n_1}|z_i|^2\right) -\ln\left(1-4\tau\sum_{i=n_1+1}^{n}|z_i|^2\right)  \right).
\end{equation}

The Fubini--Study and Bergman metrics appear as the special cases when $n_1=n$ and $n_1=0$ respectively.

\subsubsection{Rotationally invariant structures}

The \emph{rotationally invariant solutions}, meaning  K\"{a}hler potentials depending only on $\langle z,z\rangle$ rather than the individual coordinates, correspond to the case $a=0$, $X=\tau \id$ where $\tau$ is a real parameter, and $r$ is unconstrained, i.e., any real number. Potentials of rotationally invariant Bochner-K\"{a}hler structures are characterised in \cite{MR0266121} as solutions to certain differential equations, and we can directly show these $(a,X,r)$ constraints produce precisely the solutions to that equation.

Applying \eqref{eqn: examples branch 1b}, they have the form
\[
\frac{\sin 2\sqrt{r}F}{2\sqrt{r}}=\e^{-2\tau F}\langle z,z\rangle,
\]
that is
\begin{equation}\label{eqn: rotationally inv ex a}
\frac{1}{4\I \sqrt{r}} \e^{2(\tau+\I \sqrt{r}) F} - \frac{1}{4\I \sqrt{r}}\e^{2(\tau- \I \sqrt{r}) F} =\langle z,z\rangle,
\end{equation}
where $F$ is the K\"{a}hler potential, now expressed as a function of $t:=\langle z,z\rangle$. And to shorten formulas, let us write $s:=\sqrt{r}$ and $F^\prime = \tfrac{\partial }{\partial t}$ for the rest of this example, but we must stress this parameter $s$ can now be either real or purely imaginary (in contrast to the rest of this paper).

Repeatedly differentiating \eqref{eqn: rotationally inv ex a} yields
\begin{equation}
    \frac{1}{4\I s} \e^{2(\tau+\I s) F} - \frac{1}{4\I s}\e^{2(\tau-\I s) F} = t,
\end{equation}
\begin{equation}
    \frac{\tau+\I s}{2\I s} \e^{2(\tau+\I s) F} F^{\prime}- \frac{\tau-\I s}{2\I s}\e^{2(\tau-\I s) F}F^{\prime} = 1,
\end{equation}
and 
\begin{equation}
\begin{aligned}
    \frac{(\tau+\I s)^2}{\I s} \e^{2(\tau+\I s) F} (F^{\prime})^2- \frac{(\tau-\I s)^2}{\I s}\e^{2(\tau-\I s) F}(F^{\prime})^2 = & \frac{\tau-\I s}{2\I s}\e^{2(\tau-\I s) F}F^{\prime\prime} +\\
    & -\frac{\tau+\I s}{2\I s} \e^{2(\tau+\I s) F} F^{\prime\prime}.
\end{aligned}
\end{equation}

It follows that
\[
\left(\frac{(\tau+\I s)^2}{\I s} \e^{2(\tau+\I s) F}- \frac{(\tau-\I s)^2}{\I s}\e^{2(\tau-\I s) F}\right) (F^{\prime})^2 =-\frac{F^{\prime\prime}}{F^{\prime}},
\]
and the left side can be simplified by
\begin{equation}
\begin{aligned}
\frac{(\tau+\I s)^2}{\I s} \e^{2(\tau+\I s) F}- \frac{(\tau-\I s)^2}{\I s}\e^{2(\tau-\I s) F}= & 4(\tau^2+r) \left(\frac{1}{4\I s}\e^{2(\tau+\I s) F} - \frac{1}{4\I s}\e^{2(\tau-\I s) F}\right) \\
& +4\tau \left(\frac{\tau+\operatorname{i}s}{2\operatorname{i}s} \e^{2(\tau+\operatorname{i}s) F}- \frac{\tau-\operatorname{i}s}{2\operatorname{i}s}\e^{2(\tau-\operatorname{i}s) F}\right) \\
=&  -4(\tau^2+s^2)t +4\tau \frac{1}{F^{\prime}}. 
\end{aligned}
\end{equation}
This yields 
\[
4\left(\tau \frac{1}{F^{\prime}}-(\tau^2+r)t\right)(F^{\prime})^2=-\frac {F^{\prime\prime}}{F^{\prime}},
\]
which is the Tachibana--Liu differential equation (see \cite{MR0266121})
\[
F^{\prime\prime}=-4(-(\tau^2+r)tF^{\prime}+ \tau)(F^{\prime})^2= (atF^{\prime}+ k)(F^{\prime})^2,
\]
where these new $a$ and $k$ variables are matching notation in \cite{MR0266121}.

 \subsubsection{Generalised weighted projective space structures}
We return to the K\"{a}hler--Bochner metrics with potentials of the form  \eqref{eqn: examples branch 1a}, and will show that the metric
 $$\sum_{\alpha\bar{\beta}}F_{\alpha \bar{\beta}} dz_{\alpha} d\bar{z}_{\bar{\beta}}$$
 is invariant with respect to the action 
 \begin{equation}\label{act}
 (z_1,\dots,z_n) \to (\lambda_1z_1,\dots, \lambda_nz_n),
 \end{equation}
 where $\lambda_j=\e^{\frac{(\rho_j-\rho_0)c}{2}}$ for an arbitrary real constant $c$.

 Notice, that the action $z_j\to \lambda_j z_j$ in the terms $\e^{\rho_j  F}  |z_j|^2$ can be compensated by subtracting the constant $c$ from the potential $F$. However, such a change of $F$ also changes the term $\e^{\rho_0F}$, showing that the potential $F$ itself is not invariant with respect to that action.

We also have $dz_j\to \lambda_j dz_j$ and, therefore,
 $$dz_\alpha d\bar{z}_{\bar{\beta}} \to \lambda_\alpha \lambda_{\bar{\beta}}  dz_\alpha d\bar{z}_{\bar{\beta}}. $$

Hence, we need to show that 
$$F_{\alpha\bar{\beta}} \to \frac{F_{\alpha\bar{\beta}}}{{\lambda_\alpha \lambda_{\bar{\beta}}}}= \e^{\frac{-(\rho_\alpha+\rho_\beta-2\rho_0)c}{2}} F_{\alpha\bar{\beta}}.$$

We have $m_1\to \e^{-\rho_0 c}m_1$. Therefore, the first summand in $F_{\alpha\bar{\beta}}$ transforms as
$$ \frac{  \e^{-\rho_\alpha  F} \delta_{\alpha\bar{\beta}} }{ m_1} \to   \frac{  \e^{-\rho_\alpha  F} \delta_{\alpha\bar{\beta}} }{ \e^{-\rho_0  c} m_1},$$
as required, taking into account $\alpha=\bar{\beta}$.

For the second summand, we use  $m_2\to \e^{-\rho_1 c}m_2$, resulting in

\begin{multline}
  \frac{(m_2-(\rho_\alpha+\rho_\beta)w_1) \e^{-(\rho_\alpha+\rho_\beta)  F} \overline{z_\alpha}z_{\bar{\beta}}}{ m_1^3} \to \\ \frac{ \e^{-\rho_0 c}(m_2-(\rho_\alpha+\rho_\beta)m_1) \e^{-(\rho_\alpha+\rho_\beta)  F}  \e^{\frac{(\rho_\alpha+\rho_\beta-2\rho_0)c}{2}}\overline{z_\alpha}z_{\bar{\beta}}}{ \e^{-3\rho_0 c} m_1^3}.  
\end{multline}

This shows that the metric is invariant with respect to the action \eqref{act} in addition to the invariance with respect to $z_j\to \e^{\I \omega_j} z_j$ for $\omega_j\in \mathbb R$. 

If $\tau_j\ge 0$ then $\rho_j-\rho_0=\tau_j+2\tau_0>0$ and the invariance with respect to \eqref{act} can be used to extend the metric from a neighbourhood of $0$ to $\mathbb C^n$.

The following case is particularly interesting. Assume that the ratios of the $\rho_j$ are rational, that is $\rho_j=\mu_j \sigma$, where the $\mu_j$ are integers. Then by setting $V=\e^{\sigma F}$ the equation \eqref{ie} becomes
$$V^{\mu_0}+\sum_{j=1}^n  |z_j|^2 V^{\mu_j}=1,$$
that is, a rational equation on $V$, which is equivalent to a polynomial equation. If $V$ is a local positive real branch of a solution then the potential becomes
$$F=\frac{1}{\sigma} \ln V.$$

For $\tau_j=-2\tau_0=\tau$ we obtain
$$F=\frac{1}{\tau} \ln (1+\tau \sum_{j=1}^n |z_j|^2|),$$
which is the Fubiny--Study metric for $\tau>0$, and the Bergman metric for $\tau<0$.

We end up with especially nice cases whenever $r<0$ if the eigenvalues $\tau_i$ of $X$ are all rational multiples of $\tau_0$.
If $\tau_j\ge0$, this gives Bryant's description of weighted projective space structures in \cite{MR1824987}*{\S 4.4.6}. Notice that the function $s$ in \cite{MR1824987}*{\S 4.4.6} is nothing but $\e^{\rho_0 F}$, where $F$ is a K\"{a}hler potential.

\subsection{Generalised Stanton formula}
Without the previous section's $a=0$ assumption, the general defining equations become much more complicated. An open subset of the moduli space where $a\neq0$ generically can be realised, however, as rigid spheres given by the implicit generalised Stanton formula \eqref{eqn: defining eqn in Stanton coordinates}, which is expressed in terms of the parameters $(b, \theta, s)$ discussed in Section \ref{sec: potentials thm proof}. Recall that $(b, \theta, s)$, however, are neither fully normalised (as different parameter values can describe equivalent structures) nor fully general (as some Bochner-K\"{a}hler potentials cannot be described in terms of these parameters), and their relationship to the completely normalised parameters $(a,X,r)$ is summarised in Remark \ref{rem: parameter conventions}. They are useful, on the other hand, because potentials expressible in the $(b,\theta,s)$ coordinates are given by the generalised Stanton formula \eqref{eqn: defining eqn in Stanton coordinates}. For the remainder of Section \ref{sec: Examples}, we discuss new examples contained within \eqref{eqn: defining eqn in Stanton coordinates}.

\subsubsection{Stanton coordinates and the \texorpdfstring{$a=0$}{a=0} structures.}\label{Stanton coordinates and the a=0 structures.}
The class of $a=0$ structures whose general formulas are already derived in Section \ref{sec: a is 0 case} provides a nice setting to observe that the $(b,\theta,s)$ parameters cannot cover all Bochner--K\"{a}hler structures. In this section, we use the relationship between $(b,\theta,s)$ and the canonical $(a,X,r)$ parameters of Remark \ref{rem: parameter conventions} to describe the $(b,\theta,s)$ values corresponding to $a=0$ structures. We will find that not all $(X,r)$ possibilities can be obtained.

 If $a=0$ then 
\begin{equation}\label{eqn: a is zero in stanton coor}
0=(\I \theta -s\id-2\I \langle b,b\rangle\id)b.
\end{equation}
Normalised $(a,X,r)$ values with $a=0$ and $r\geq 0$ are all realised in the $(b,\theta,s)$ parameters by $(b,\theta,s) = (0,X,\sqrt{r})$.
But the solutions to \eqref{eqn: a is zero in stanton coor} with $b=0$ produce $r=s^2\geq 0$, so these cannot realise $(a,X,r)$ values with $a=0$ and $r<0$. Hence, to realise other $(X,r)$ pairs, we must find solutions with $b\neq 0$.

By multiplying \eqref{eqn: a is zero in stanton coor} by $b^*$ from the left, we get
\[
\I (b^*\theta b -2\langle b,b\rangle b^*b) -sb^*b=0.
\]
It follows $s=0$ because $\theta$ is hermitian and $s\in \mathbb{R}$, and hence
\[
\theta b= 2\langle b,b\rangle b.
\]
That is, $b$ is an eigenvector of $\theta$ with eigenvalue $2\langle b,b\rangle$. By applying the standard normalisation of Hermitian matrices, we may assume that $\theta$ is diagonal  and $b$ is a column vector in the form $(\beta,0,\ldots, 0)$ for some $\beta\in \mathbb{C}$. The resulting $X$ and $r$ are
$$X=\theta-2bb^*-\langle b,b\rangle \id=\begin{bmatrix} -|\beta|^2& 0& \dots & 0\\0&\theta_2-|\beta|^2&\dots & 0\\\vdots& \vdots& \dots &0\\ 0&0&\dots& \theta_n-|\beta|^2 \end{bmatrix}$$
and 
\[
r=s^2-2b^*\theta b+ 3 \langle b,b\rangle= -|\beta|^2<0,
\]
where $\theta_j$ are diagonal entries in $\theta$ (and hence $\theta_1=2\langle b,b\rangle$), which achieves negative $r$. The only restriction on $X$ is the presence of a negative eigenvalue, equal to $r$. This restriction is invariant under the minor normalisation we applied to assume $\theta$ is diagonal, because that is achieved via a $\operatorname{GL}(\mathbb{C}^n)$ action that induces natural hermitian forms transformations of $X$. 

The corresponding Stanton formula is
$$
    \frac{1-\e^{-4|\beta|^2v}}{4|\beta|^2}=\frac{\langle \e^{-2\theta v}z,z\rangle}{|1-2\I z_1\bar{\beta}|^2}+\frac{(\e^{-4|\beta|^2v}-1) (\I(z_1\bar{\beta}-\bar{z_1}\beta)-4|z_1|^2|\beta|^2)}{2|\beta|^2|1-2\I z_1\bar{\beta}|^2}
$$
which further simplifies to
$$\frac{1-\e^{-4|\beta|^2v}}{4|\beta|^2}(1-4|\beta|^2|z_1|^2)=\frac{4|\beta^2|\langle \e^{-2\theta v}z,z\rangle }{4|\beta|^2}$$
\[
\frac{\e^{2|\beta|^2v}-\e^{-2|\beta|^2v}}{4|\beta|^2}=\frac{\sinh(2|\beta|^2v)}{2|\beta|^2}=\frac{\sinh(2rv)}{2r}=\langle \e^{-2X v}z,z\rangle,
\]
matching \eqref{eqn: examples branch 1a}.

Since this is a general solution of all $(b,\theta,s)$ realising $(a,X,r)$ with $a=0$ and $r<0$, comparing with Section \ref{sec: homogeneous structures} provides a nice proof of the following lemma.
\begin{lem}
The Bergman metrics, which have potentials of the form
\[
v =F(z,\bar{z})= -\frac{1}{4\tau}\ln\left(1-4\tau\sum_{i=1}^{n}|z_i|^2\right),
\]
cannot be realised by the generalised Stanton formula \eqref{eqn: defining eqn in Stanton coordinates}.
\end{lem}

\subsubsection{New closed form examples}\label{sec: new closed form examples}

For the cases with $b\neq 0$, $\theta=0$, and $s=0$, the potentials are defined by setting $\theta=0$ in \eqref{eqn: defining eqn in Stanton coordinates} and taking the limit as $s$ tends to $0$, yielding
\[
0 = 2\langle b,b\rangle v^2 -\left(1+\frac{2\I \langle z,b\rangle}{1-2\I\langle z,b\rangle}-  \frac{2\I \langle b,z\rangle}{1+2\I\langle b,z\rangle}\right)v+\frac{\langle z,z\rangle}{|1+2\I\langle b,z\rangle|^2}
\] 
or, equivalently, 
\[
2|1+2\I\langle b,z\rangle|^2 \langle b,b\rangle v^2 - (1-4|\langle z,b\rangle|^2) v+ \langle z,z\rangle.
\]
Hence, the potential $v$ can be explicitly described as
\[
v=\frac{1-4|\langle z,b\rangle|^2 - \sqrt{(1-4|\langle z,b\rangle|^2)^2- 8 |1+2\I\langle b,z\rangle|^2 \langle b,b\rangle  \langle z,z\rangle}}{4|1+2\I\langle b,z\rangle|^2\langle b,b\rangle
}
\]
for any $b\neq 0$. A unique solution is determined because only one of the two solutions to this quadratic equation satisfies $v(0)=0$.

The corresponding $a,X,r$ parameters are
\begin{equation}
X= \langle b,b\rangle -2bb^*,
\quad
a=-2\I \langle b,b\rangle b,
\quad\text{ and }\quad
r=-3\langle b,b\rangle^2.
\end{equation}

The infinitesimal symmetries are vector fields of the form
\[
\begin{aligned}
    X = & (\beta+(s+\I H)z+\alpha w+2\I \langle z, \alpha \rangle z + \rho zw) \frac{\partial}{\partial z}+ \\
    & + (2\I \langle z,\beta\rangle + 2sw+ 2\I \langle z,\alpha \rangle w+ \rho w^2)\frac{\partial}{\partial w}
\end{aligned}
\]
satisfying
\begin{equation}
\left[X, \\b \frac{\partial}{\partial z}+ (1+ 2\I \langle z,b\rangle )\frac{\partial}{\partial w}\right]=0
\end{equation}
or, equivalently,
\[
\begin{aligned}
0=&\left[\begin{bmatrix}0 & -2\I \alpha^* &-\rho \\ \beta & \I H +s &  \alpha\\ 0 & 2\I \beta^* & 2s  \end{bmatrix}, \begin{bmatrix}0 & 0 &0 \\ b & 0 &  0\\ 1 & 2\I b^* & 0  \end{bmatrix}\right]\\
=&\begin{bmatrix}-2\I \langle b,\alpha\rangle-\rho & -2\I\rho b^*&0\\ (s+\I H)b+\alpha & 2\I \alpha b^* -2\I b\alpha^* &  \rho b\\ 2s+ 2\I \beta^*b-2\I b^* \beta  & 4\I sb^*- 2\I \alpha^*+2\I b^*(s+\I H)& -2\I \langle \alpha,b\rangle  \end{bmatrix}.
\end{aligned}
\]

Since $b\neq 0$, it immediately follows that $\rho=0$. Furthermore,
$$ \langle \alpha,b\rangle=0.$$
Now, multiplying both sides by $b$ in
\[
0=2\I \alpha b^* -2\I b\alpha^*
\]
gives
$$2\I \alpha \langle b,b\rangle  -2\I b\langle b,\alpha\rangle=0$$
and hence 
$$\alpha=0.$$
Since $\I H$ is skew-hermitian and has imaginary eigenvalues, the eigenvalue equation
$$\I Hb=-sb$$
can only hold with $Hb=0$ and $s=0$. This gives an $(n-1)^2$ parametric family of symmetries that correspond to an $U(n-1)$ rotation. 
Finally,
$$2s+ 2\I \beta^*b-2\I b^* \beta =0$$
implies
$$\Im \langle \beta,b\rangle=0.$$
This gives a complex $n^2$-dimensional family of infinitesimal automorphisms
$$(\beta+\I Hz) \frac{\partial}{\partial z}+ 2\I \langle z,\beta\rangle \frac{\partial}{\partial w}$$
with $\beta\perp_{\mathbb{R}} \I b$ and $Hb=0$.

\section{Infinitesimal Symmetry Algebras}\label{sec: ISA}

In this section we prove Theorem 2 on symmetries of Bochner--K\"{a}hler manifolds, and as a consequence, classify the Bochner--K\"{a}hler manifolds according to their symmetries.
\begin{theorem}\label{thm: infinitesimal symmetries}
Let $M$ be the (local) Bochner--K\"{a}hler manifold associated with the parameters $a$, $X$, and $r$ of Theorem \ref{thm: general potentials}. The infinitesimal automorphisms of $M$ correspond to the vector fields
\[(\beta+\I \hBlock z+\alpha w+2\I \langle z,\alpha\rangle z+\rho zw)\frac{\partial}{\partial z}+(2\I \langle z,\beta \rangle + 2\I \langle z,\alpha\rangle w+\rho w^2)\frac{\partial}{\partial w},\]
where 
\[
    \alpha=\I X \beta 
    \quad\mbox{ and }\quad\rho =2\Im \langle \beta,a\rangle
\]
and the Hermitian matrix $\hBlock$ and $\beta\in\mathbb{C}^n$ satisfy 
\[
    [\hBlock,X]=-2\I(a\beta^*+\beta a^*)
    \quad\mbox{ and }\quad
    \I \hBlock a=(X^2+r)\beta.
\]
\end{theorem}

\begin{proof}
The infinitesimal automorphisms of the Sasakian manifold (and hence the K\"{a}hler manifold) are the vector fields in $\mathfrak{su}(n+1,1)$ that commute with $\reeb$, that is, the centraliser of $\reeb$ in $\mathfrak{su}(n+1,1)$. Applying the normalisations of Section \ref{sec: moduli and symmetries}, we will work in coordinates where $\reeb$ is given by \eqref{Reeb VF normalised} with parameters $a$, $X$, and $r$, normalised as in Proposition \ref{prop2}

Using the matrix representation \eqref{indef su rep}  (with the convention $c=\I S$) of $\mathfrak{su}(n+1,1)$, $\reeb$ is represented by the matrix \eqref{eqn: potential parameters}.  We will use labels $\beta,\alpha\in\mathbb C^n$, $\hBlock$ - hermitian, and $\rho,s\in \mathbb R$, comprising the matrix
\begin{equation}\label{indef su rep symmetry}
\begin{bmatrix}
0 & -2\I \alpha^* & \rho \\ 
\beta & \I \hBlock +\sigma &  \alpha\\ 
0 & 2\I \beta^* & 2\sigma
\end{bmatrix},
\end{equation}
for components of a general infinitesimal symmetry modulo $\operatorname{span}\{\reeb\}$. So $(a,X,r)$ parameterises the possible K\"{a}hler structures while $(\alpha,\beta,\hBlock, \sigma, \rho)$ parameterises relevant CR symmetries.

The centraliser of $\reeb$ is given by
\begin{multline*}
0=\left[\begin{bmatrix}0 & -2\I \alpha^* &\rho \\ \beta & \I \hBlock +\sigma &  \alpha\\ 0 & 2\I \beta^* & 2\sigma  \end{bmatrix}, \begin{bmatrix}0 & -2\I a^* &r \\ 0 & \I X &  a\\ 1 & 0 & 0  \end{bmatrix}\right]=\\\begin{bmatrix}+\rho+2\I \langle \beta,a\rangle & 2 \alpha^*X-2a^* \hBlock+2\I (\sigma a^*- r\beta^*)&-2\I\langle a,\alpha\rangle +2\I\langle \alpha,a\rangle-2r\sigma\\ \alpha-\I X\beta & [X,\hBlock]-2\I \beta a^*-2\I a\beta^* &  r\beta-\sigma a+\I \hBlock a-\I X\alpha\\ 2\sigma & -2 \beta^*X+ 2\I \alpha^*& -\rho+2\I \langle a,\beta\rangle  \end{bmatrix},
\end{multline*}
which yields
\begin{equation}\label{aut1}
[\hBlock,X]=-2\I (a\beta^*+\beta a^*),
\end{equation}
\begin{equation}\label{aut2}
-\I \hBlock a=(X^2+r\id)\beta,
\end{equation}
\begin{equation}\label{aut3}
\sigma=0,
\quad
\alpha= \I X\beta,
\quad
\rho=2\Im \langle \beta,a\rangle,
\end{equation}
\begin{equation}\label{aut4}
\Re\langle \beta,a\rangle=0
\quad\text{ and }\quad
\Im \langle \alpha, a\rangle=0.
\end{equation}
In addition to $\sigma=0$, from \eqref{aut3} both $\alpha$ and $\rho$ are uniquely determined by $\beta$ and the given data $X,a$.

Next, we analyse equation \eqref{aut1}.  Given the Proposition \ref{prop2} normalisation, $X$ is diagonal with real entries $x_1\dots,x_n$.  Accordingly,
\[
[\hBlock,X]=
\begin{bmatrix} 
0& (x_2-x_1)h_{12} & \cdots & (x_n-x_1) h_{1n}\\ 
(x_1-x_2) \overline{h_{12}}& 0& \cdots &(x_n-x_2)h_{2n}\\ 
\vdots &\vdots& \ddots & \vdots \\ 
(x_1-x_n) \overline{h_{1n}}& (x_2-x_n) \overline{h_{2n}}& \cdots & 0 
\end{bmatrix}
\]
which equals
\[
-2\I (a \beta^*+ \beta a^*)=
-2\I \begin{bmatrix} a_1 (\beta_1+\bar{\beta}_1)& a_1 \bar{\beta}_2+a_2\beta_1&\cdots& a_1 \bar{\beta}_n+a_n\beta_1 \\ a_2 \bar{\beta}_1+a_1\beta_2& a_2 (\beta_2+\bar{\beta}_2)& \cdots & a_2 \bar{\beta}_n+a_n\beta_2\\\vdots & \vdots & \ddots & \vdots\\ a_n \bar{\beta}_1+a_1\beta_n& a_n \bar{\beta}_2+a_2\beta_n& \cdots& a_n (\beta_n+\bar{\beta}_n)
\end{bmatrix}.
\]

The entries on the main diagonal yield
\begin{equation}\label{aut7}
a_j(\beta_j+\bar{\beta}_j)=0
\end{equation}
for all $j$.  Hence, either $a_j=0$ or $\beta_j$ is purely imaginary. Accordingly $\Re \langle X\beta, a\rangle=\Re \langle X\beta, a\rangle=0$, from which \eqref{aut3} implies the other condition in \eqref{aut4}. Thus conditions \eqref{aut1} through \eqref{aut3} alone are sufficient for the CR symmetry to preserve $\reeb$. 
\end{proof}

\begin{cor}\label{cor: homogeneous structures}
    A Bochner--K\"{a}hler manifold $M$ is locally homogeneous if and only if the parameter $a=0$ and the eigenvalues $\tau$ of $X$ satisfy $\tau^2=-r$. 
    
    For $\tau=0$, we get the flat K\"{a}hler metric on $\mathbb C^n$.
    
    For $\tau\neq 0$, the locally defined manifold extends to a Cartesian product of a $\mathbb CP^{n_1}$ with Fubini--Study metric and a ball $\mathbb B^{n_2}$ with Bergman metric, with potential
    \begin{equation}\label{eqn: T3 defining funciton}
    v=\frac{1}{\tau}\log (1+\tau \langle z',z'\rangle)- \frac{1}{\tau}\log (1-\tau \langle z'',z''\rangle)
    \end{equation}
    where $n_1+n_2=n$ and $\tau>0$ is a real parameter. These K\"{a}hler manifolds are globally homogeneous. In this case the parameter $\beta\in \mathbb C^n$ is arbitrary and $\hBlock$ consists of two arbitrary hermitian diagonal blocks of dimensions $n_1\times n_1$ and $n_2\times n_2$, respectively, resulting in a symmetry group of dimension $2n+n_1^2+n_2^2$.  
\end{cor}

\begin{proof}
For local homogeneity the infinitesimal automorphisms must span the tangent space at the reference point. This occurs if and only if there is no restriction on $\beta$, that is, if and only if $a=0$ and $X^2=-r \id$. Regarding \eqref{eqn: T3 defining funciton}, it coincides with \eqref{eqn: examples branch homogeneous}, derived in  \S\ref{sec: homogeneous structures}.
\end{proof}

The solutions to equations \eqref{aut1} and \eqref{aut2} depend on the structure of eigenspaces of $X$. In the following, we investigate the resulting options.

\begin{prop}\label{prop: symmetry algebras}
    Let $V_1,\dots,V_{\mu}$ be the eigenspaces with eigenvalues $\tau_1,\dots,\tau_{\mu}$ of $X$. Let $a_{V_k}$ be the components of $a$ with respect to the eigenspace decomposition of $X$. Then,
    \begin{enumerate}
        \item The blocks $\hBlock_{ij}$ of the matrix $\hBlock$ that correspond to different eigenspaces $V_i$ and $V_j$ are uniquely determined by a choice of $\beta$.
        \item If $a_{V_k}\neq 0$ then $\beta_{V_k}=\I \lambda_k a_{V_k}$ where $\lambda_k$ is a real parameter and $H_{kk}$ satisfies $n_k$ equations, which reduces its freedom to $(n_k-1)^2$. This contributes $(n_k-1)^2+1$ parameters.
        \item If $a_{V_k}= 0$ then $H_{kk}$ is unrestricted and $\beta_{V_k}$ is either unrestricted, or vanishing.  
        This contributes  either $n_k^2$, or $n_k^2+2n_k$ parameters.
    \end{enumerate}
\end{prop}

{\bf Proof.} If two eigenvalues $x_i,x_j$ ($i<j$) of $X$ are different, then, by \eqref{aut1}, the entry $h_{ij}$ of the matrix $\hBlock$ is determined by 
\begin{equation}\label{aut8}
(x_j-x_i)h_{ij}=-2\I(a_i \bar{\beta}_j+a_j\beta_i).
\end{equation}
This proves the first statement.

For the proof of the second statement, assume that $x_i=x_j=\tau_k$. Then $h_{ij}$ cannot be determined from equation \eqref{aut1} and 
\[
a_i \bar{\beta}_j+a_j\beta_i=0
\qquad\forall\, i,j.
\]
For some $j$, we have $a_j\neq 0$, implying $\bar{\beta}_j=-\beta_j$, so for any such $j$ the previous equation becomes
\[
\begin{vmatrix} a_i & a_j\\\beta_i& \beta_j\end{vmatrix} = \begin{vmatrix} a_i & a_j\\\beta_i& -\overline{\beta_j}\end{vmatrix}
=0 
\quad\quad\forall\, i.
\]
If $a_{V_k}\neq 0$, this implies that $\beta_{V_k}$ is a multiple of $a_{V_k}$, and since $\bar{\beta}_j=-\beta_j$ it follows that the multiplier is purely imaginary, that is
\[\beta_{V_k}=\I \lambda_k a_{V_k}\]
where $\lambda_k$ is a real parameter.

The hermitian matrix $H_{kk}$ is subject to the equation \eqref{aut2} restricted to $V_k$
\begin{equation}\label{Hkk}
-\I H_{kk}a_{V_k}=\I \sum_{j\neq k} H_{kj}a_{V_j}+(\tau_k^2+r)\beta_{V_k}.  
\end{equation}

If $a_{V_k}\neq 0$, equation \eqref{Hkk} determines one linear combination of the rows of $H_{kk}$ for a choice of $\beta$, thus leaving $(n_k-1)^2$ parameters. 

If $a_{V_k}=0$, there is no restriction on $H_{kk}$ but 
$$\I \sum_{j\neq k} H_{kj}a_{V_j}+(\tau_k^2+r)\beta_{V_k}=0.$$

By \eqref{aut1}, we have

\begin{equation}\label{hjk}
  h_{lj}=-\frac{2\I}{\tau_k-\tau_j}\beta_l a_j
\end{equation}
for all $l \in I_{V_k}$ and $j \notin I_{V_k}$, where $I_{V_k}$ is the index set corresponding to $V_k$.
By \eqref{aut2}, we get

\begin{equation}
 -\I (H a)_l = -\I \sum_{j\notin I_{V_k}}h_{lj}a_j = (\tau_k^2 + r) \beta_l.
\end{equation}

Hence,  by substituting from \eqref{hjk},
\begin{equation} \label{suma2}
-2 \beta_l \sum_{j\notin I_{V_k}} \frac{a_j^2}{\tau_k-\tau_j}  =   (\tau_k^2 + r) \beta_l. 
\end{equation}

It follows that for $\tau_j, a_j$ and $r$ satisfying 
\begin{equation} \label{suma3}
2 \sum_{j\notin I_{V_k}} \frac{a_j^2}{\tau_k-\tau_j}  +  (\tau_k^2 + r) =0,
\end{equation}
$\beta_{V_k} $ can be arbitrary. If \eqref{suma3} does not hold, then $\beta_{V_k} = 0$.

\hfill $\Box$ \medskip

Let $n_k$ be again the multiplicity of the $k$-th eigenvalue of $X$ and let $I_0$ denote the set of indices $j$ such that $a_{V_j} = 0$. Let $I_1$  be the set of indices $j$ such that $a_{V_j} \neq 0$ and $I_S$ the set of indices in $I_0$ such that \eqref{suma2} holds.

By Proposition 3, we obtain the following formula for the dimension of the symmetry algebra,

\begin{equation}
 \sum_{k \in I_0} n_k^2 + \sum_{k \in I_1} ((n_k-1)^2 + 1) + 2\sum_{k \in I_S}n_k,
\end{equation}

Note that this reproves the formula derived by Bryant (p.648 in \cite{MR1824987}).
An immediate consequence of this formula is the following result, which determines the submaximal dimension of the symmetry algebra.
Recall that the maximal dimension is equal to $n^2 + 2n$.

\begin{prop}\label{prop: submax}{The submaximal dimension of the symmetry algebra occurs for a homogeneous manifold from Corollary 1, with 
$n_1 = 1.$ The dimension is equal to $n^2 + 2$.}
\end{prop}

{\bf Example.}  \iffalse Consider \GS{I hoped to be able to produce an example with one $p$ unrestricted, but this may not work.}
\[X=\begin{bmatrix}
    -3&0\\0&0
\end{bmatrix}, \quad a=\begin{bmatrix}
    -\I\sqrt{2}\\0
\end{bmatrix}, \quad r=0. \]

Then
\[H=\begin{bmatrix}
    -\frac{9}{\sqrt{2}}p_1&0\\0&h_{22}
\end{bmatrix}, \quad p=\begin{bmatrix}
    p_1\\0
\end{bmatrix},\quad \alpha=\begin{bmatrix}
    -3\I p_1\\0
\end{bmatrix}, \quad \rho=-6\sqrt{2} p_1, \]
where $p_1, h_{22}$ are real, resulting in a 2-parametric family.

This corresponds to
\[\theta=\begin{bmatrix}
    3&0\\0&2
\end{bmatrix}, \quad b=\begin{bmatrix}
    \sqrt{2}\\0
\end{bmatrix}, \quad R=0. \]

Since
\begin{align*}
    X=&\begin{bmatrix}
    -3&0\\0&0
\end{bmatrix}=\theta-2bb^*-b^*b=\begin{bmatrix}
    3&0\\0&2
\end{bmatrix}-2\begin{bmatrix}
    2&0\\0&0
\end{bmatrix}-\begin{bmatrix}
    2&0\\0&2
\end{bmatrix}\\
a=&\begin{bmatrix}
    -3\I\sqrt{2}\\0
\end{bmatrix}=(\I\theta-R-2\I \langle b,b\rangle)b=\I \begin{bmatrix}
    -1&0\\0&-2
\end{bmatrix}\begin{bmatrix}
    \sqrt{2}\\0
\end{bmatrix}=\begin{bmatrix}
    -\I\sqrt{2}\\0
\end{bmatrix}\\
r=&-R^2+2b^*\theta b-3\langle b,b\rangle^2=0
\end{align*}
\fi

Consider the case
\[a=0, \quad r=0. \]
Then, according to \eqref{eqn: stanton to general normal coordinates}, the Stanton parameter $\theta$ equals the matrix parameter $X$, resulting in the Stanton equation. 

\[v=\langle \e^{-2\theta v}z,z\rangle.\]
Then the infinitesimal automorphisms have the form 
\[(\beta+\I Hz)\frac{\partial}{\partial z}\]
with $[H,\theta]=0$ and $\theta \beta=0$. 

Assume the $\theta$ is diagonal and has $n_1$ eigenvalues $0$. Then
$$v=\langle z',z'\rangle + \langle \e^{-2\theta''v}z'',z''\rangle,$$
where $z'=(z_1,\dots,z_{n_1})$ and $z''=(z_{n_1+1},\dots,z_{n})$ and $n_2,\dots,n_r$ are the multiplicities of the non-zero eigenvalues of $\theta$.  Then the dimension of the symmetry group is $2n_1+n_1^2+\cdots+n_r^2$. The automorphisms are
\begin{align*}
\zeta'=&U_1z'+\beta\\
\zeta''=&\e^{\theta''(2\langle U_1z',\beta \rangle+\langle \beta,\beta\rangle)} U_2 z''\\
\omega=& w+2\I \langle U_1z',\beta\rangle +\I\langle \beta,\beta\rangle
\end{align*}
where $U_1$ is a unitary matrix and $U_2$ is a matrix with unitary blocks of dimensions $n_2,\dots,n_r$.

\subsection{Complete classification for complex surfaces (\texorpdfstring{$n=2$}{n=2})}\label{complex surface case}

For $n=2$ we have two cases: a) $\tau=x_1=x_2$ and b) $x_1=\tau_1\neq x_2=\tau_2$. 

Consider case a). Without loss of generality $a_1\ge 0$ and $a_2=0$. If $a_1=0$ there is no restriction on $H$ and the only restriction on $\beta$ is
\[
(\tau^2+r)\beta=0.
\]
Thus, either $\tau^2=-r$ and $\beta$ is arbitrary (This is the Fubini--Study or Bergman space with 8-dimensional symmetry), or $\tau^2\neq -r$ and $\beta=0$ (4-dimensional symmetry). An example of this is the Bochner--K\"{a}hler manifold with potential $v$ that satisfies the implicit equation
$$v=\e^{-\tau v}\langle z,z\rangle$$
or the explicit equation
$$v=\frac{1}{\tau}W(\tau \langle z,z\rangle),$$
where $W$ is the Lambert $W$ function. The symmetry group is $U(2)$.
\medskip

If $a_1\neq 0$ it follows that $\beta_1$ is imaginary, $\beta_2=0$, $h_{12}=0$, and
$$h_{11}=\frac{(\tau^2+r)\beta_1}{\I a_1}.$$
This results in 2 symmetries.

In case b) we know that $\beta_j$ is purely imaginary or $a_j=0$ and
$$h_{12}=\frac{2\I}{\tau_1-\tau_2}(a_1\bar{\beta_2}+a_2\beta_1).$$

Furthermore,
\begin{align*}
h_{11}a_1+ h_{12}a_2&=\I(\tau_1^2+r)\beta_1\\
\overline{h_{12}}a_1+ h_{22}a_2&=\I(\tau_2^2+r)\beta_2.
\end{align*}

If $a_1=a_2=0$ then $h_{12}=0$. We have three options: if both $\tau_j^2\neq -r$ then 
$\beta_j=0$ (2-dimensional symmetry) and if one or both $\tau_j^2=- r$ the corresponding $\beta_j$ is arbitrary (4 or 6-dimensional symmetry). The case $\tau_1^2=-r$ corresponds to the manifold with infinitesimal symmetries (of the Heisenberg sphere)
\[(\beta+h_{11}z_1 + \I\tau_1\beta w+ 2\tau_1\bar{\beta}z_1^2)\frac{\partial}{\partial z_1}+ (h_{22}z_2+2\tau_1\bar{\beta}z_1z_2)  \frac{\partial}{\partial z_2}+ (2\I \bar{\beta}z_1+2\tau_1\bar{\beta}z_1w))\frac{\partial}{\partial w},\]
where $\beta\in \mathbb C$ and $h_{11},h_{22}\in\mathbb R$ are arbitrary.

The implicit defining equation (in normal coordinates) is
\[\frac{\sinh 2\tau_1 v}{2\tau_1}=\e^{-2\tau_1 v}|z_1|^2+ \e^{-2\tau_2 v}|z_2|^2\]
and the automorphisms have the form
\begin{align*}
    z_1'&=\frac{1}{2\sqrt{\tau_1}}\e^{\I \theta_1}\frac{2\sqrt{\tau_1}z_1+b}{1-2\sqrt{\tau_1}\bar{b}z_1}\\
    z_2'&=\e^{\I \theta_2} \frac{(1+|b|^2)^\frac{\omega}{2} z_2}{(1-2\sqrt{\tau_1}\bar{b}z_1)^\omega},
\end{align*}
where $\omega= \frac{\tau_1+\tau_2}{2\tau_1}$.

The Bochner--K\"{a}hler manifold with 6-dimensional group  is the Cartesian product of $\mathbb{CP}^1$ with Fubini--Study and a Poincar\'e--Lobachevsky disc (Bergman metric), which is homogeneous.\medskip

If both $a_1\neq 0$ and $a_2\neq 0$ then $\beta_1$ and $\beta_2$ are purely imaginary and
\begin{align*}
    h_{12}&=\frac{2\I}{\tau_1-\tau_2}( -a_1\beta_2+a_2\beta_1)\\
    h_{11}&=-h_{12}\frac{a_2}{a_1}+ \frac{\I (\tau_1^2+r)\beta_1}{a_1}\\
    h_{22}&=-\overline{h_{12}}\frac{a_1}{a_2}+ \frac{\I (\tau_2^2+r)\beta_2}{a_2}.
\end{align*}
This results in a 2-dimensional symmetry algebra.

If either $a_1=0$ or $a_2=0$, say $a_1=0$, we obtain two cases, according to Proposition 3. If 
\begin{equation} \label{ataur}
    a_2^2 + (\tau_2 - \tau_1)(\tau_1^2 + r) = 0, 
\end{equation}

then $\beta_1$ is arbitrary and $\beta_2$ is purely imaginary, $h_{11}$ is arbitrary (real) and 
\begin{align*}
    h_{12}&=\frac{2\I a_2\beta_1} {\tau_1-\tau_2}\\
    h_{22}&= \frac{\I (\tau_2^2+r)\beta_2}{a_2},
\end{align*}
resulting in a 4-dimensional symmetry.
If \eqref{ataur} does not hold, then $\beta_1 =h_{12} = 0$, and the symmetry algebra is 2-dimensional.

Beyond counting the algebra dimensions via counting unconstrained variables remaining at the end of each branch in this branching analysis, we can furthermore substitute the variable constraints into the symmetry formula in Theorem \ref{thm: infinitesimal symmetries} and determine the symmetry algebra's isomorphism class with standard Lie theory techniques. 
We carried out such analysis with the computer software \emph{dgcv} \cite{dgcv}, we report its results in Table \ref{tab: surface moduli space stratification}, which presents the complete table of symmetry algebras computed for every germ of Bochner-K\"{a}hler structures on complex surfaces, expressed in the parameters of the canonical  form of Theorem \ref{thm: general potentials}.

\begin{table}[hbt!]
    \centering
    \begin{tabular}{|c:c|c|c:c|l|}
        \hline
         \multicolumn{2}{|c|}{\shortstack[c]{strata defining conditions\\(defined up to\\ permutation of indices)}} & \shortstack[c]{symmetry\\algebra} & \multicolumn{2}{|c|}{\shortstack[c]{strata \&\\ algebra\\ dimensions}} & notes\\\hline\hline
         \multirow{3}{*}{\parbox{2.8cm}{\flushleft $a = 0$, $\tau_1=\tau_2$, $r=-\tau_1^2$}} & $\tau_1 = 0$ & $\mathbb{C}^2\rtimes\mathfrak{u}(2)$ & \parbox{.7cm}{\centering 0} & \multirow{3}{*}{8} & Euclidean \\
         \cline{2-4}\cline{6-6}
         & $\tau_1 > 0$ & $\mathfrak{su}(3)$ & 1 & & Fubini--Study\\\cline{2-4}\cline{6-6}
         & $\tau_1 < 0$ & $\mathfrak{su}(2,1)$ & $1$ & & Bergman metric\\
         \hline
         \parbox{2.8cm}{\flushleft $a=0$, $\tau_1=-\tau_2$, $r=-\tau_1^2$ }& $\tau_1\neq 0$ & $\mathfrak{su}(2)\oplus \mathfrak{su}(1,1)$ & 1 & 6 &\parbox{2.5cm}{\flushleft Tachibana--Liu\\ solutions \cite{MR0266121}}\\
         \hline  
         \multicolumn{2}{|c|}{$a=0$, $\tau_1=\tau_2$, $r\neq-\tau_1^2$} & $\mathfrak{u}(2)$ & 2 & & 4-d. isotropy\\
         \cline{1-4}\cline{6-6}
         \multirow{3}{*}{\parbox{2.8cm}{\flushleft $a=0$, $\tau_1\neq\pm\tau_2$, $r=-\tau_1^2$}} & $\tau_1=0$ & algebra \eqref{eqn: kahler surface aut solvable mat algebra} & 1 & & \multirow{3}{*}{\parbox{2.4cm}{\flushleft 2-dimensional\\ isotropy}} \\
         \cline{2-4}
         & $\tau_1 < 0$ & $\mathfrak{u}(1,1)$ & 2 & &\\
         \cline{2-4}
         & $\tau_1 > 0$ & $\mathfrak{u}(2)$ & 2 & \multirow{5}{*}{4} &\\
         \cline{1-4}\cline{6-6}
         \multirow{3}{*}{\parbox{2.8cm}{\flushleft $a_2=0$, $\tau_1=\frac{-2a_1^2 + \tau_2(r+\tau_2^2)}{r+\tau_2^2}$, $a_1\neq 0$, $\tau_1\neq \tau_2$,  $r\neq -\tau_2^2$}}& \parbox{2.2cm}{\centering $\,$ \\$\tau_2< 0$, $a_1 = \pm \frac{(r+\tau_2^2)}{2\sqrt{-\tau_2}}$ \\ $\,$ }& algebra \eqref{eqn: kahler surface aut solvable mat algebra} & 2 & & \multirow{3}{*}{\parbox{2.4cm}{\flushleft 1-dimensional\\ isotropy}}\\
         \cline{2-4}
         & \vphantom{$\Bigg($}\parbox{2.2cm}{\centering  $a_1^2 < \frac{(r+\tau_2^2)^2}{-4\tau_2}$}& $\mathfrak{u}(2)$ & 3 & & \\
         \cline{2-4}
         & \vphantom{$\Bigg($}\parbox{2.2cm}{\centering $a_1^2 > \frac{(r+\tau_2^2)^2}{-4\tau_2}$}& $\mathfrak{u}(1,1)$ & 3 & & \\
         \hline
         \multicolumn{2}{|c|}{$a_2=0$, $\tau_1=\tau_2$, $a_1\neq 0$}& \multirow{3}{*}{$\mathbb{R}^2$} & 3 & \multirow{3}{*}{2} & \multirow{3}{*}{trivial isotropy}\\
         \cline{1-2}\cline{4-4}
        \multicolumn{2}{|c|}{$a_2=0$, $r= -\tau_j^2$, $a_1\neq 0$, $\tau_1\neq\tau_2$}& & 3 & & \\
         \cline{1-2}\cline{4-4}
         \multicolumn{2}{|c|}{otherwise} & & 5 & & \\
         \hline
    \end{tabular}
    \caption{Stratification of the moduli space of germs of Bochner--K\"{a}hler complex surfaces into sets where the symmetry algebra is constant. Strata are described by conditions on the canonicalised moduli space parameters $(a,X,r)$ of Theorem \ref{thm: general potentials} with $\tau_1,\tau_2$ being eigenvalues of $X$. Permuting indices in a strata's defining conditions describes another strata with equivalent properties, which are omitted from the table for brevity.
    }
    \label{tab: surface moduli space stratification}
\end{table}

Most of the symmetry algebras in Table \ref{tab: surface moduli space stratification} have standard labels.  There is one exception, and we will label this additional algebra $\mathfrak{s}_4$; it is the solvable algebra having matrix representation for which there is a basis $(e_0, e_1,e_2, e_3)$ satisfying
\begin{equation}\label{eqn: kahler surface aut solvable mat algebra}
\sum_{j=0}^3x_je_j =
\begin{bmatrix}
0 & 0 & 0 & 0 \\
x_1 & 0 & x_0 & 0 \\
x_2 & -x_0 & 0 & 0 \\
x_3 & -x_2 & x_1 & 0 \\
\end{bmatrix}.
\end{equation}

\subsection{Complete classification for complex 3-manifolds (\texorpdfstring{$n=3$}{n=3})}\label{complex 3-manifold case}
A complete calculation of symmetry algebras on Bochner--K\"{a}hler 3-manifolds can be obtained with the same techniques demonstrated in \S\ref{complex surface case}. Tables \ref{tab: 3fold moduli space stratification} and \ref{tab: 3fold moduli space stratification low dim symmetries} present a complete stratification of the moduli space of germs of Bochner-K\"{a}hler structures on 3-manifolds such that corresponding K\"{a}hler symmetry algebra isomorphism classes are constant along the strata.

\begin{table}[hbt!]
    \centering
    \begin{tabular}{|c:c|c|c:c|l|}
        \hline
         \multicolumn{2}{|c|}{\shortstack[c]{strata defining conditions\\(defined up to\\ permutation of indices)}} & \shortstack[c]{symmetry\\algebra} & \multicolumn{2}{|c|}{\shortstack[c]{strata \&\\ algebra\\ dimensions}}& notes\\\hline\hline
         \multirow{3}{*}{\parbox{3cm}{\flushleft $a = 0$, $\tau_1=\tau_2=\tau_3$, $r=-\tau_1^2$}} & $\tau_1 = 0$ & $\mathbb{C}^3\rtimes\mathfrak{u}(3)$ & \parbox{.7cm}{\centering 0} & \multirow{3}{*}{15} & Euclidean \\
         \cline{2-4}\cline{6-6}
         & $\tau_1 > 0$ & $\mathfrak{su}(4)$ & 1 &  & Fubini--Study\\
         \cline{2-4}\cline{6-6}
         & $\tau_1 < 0$ & $\mathfrak{su}(3,1)$ & 1 &  & Bergman metric\\
         \hline
         \multirow{2}{*}{\parbox{3cm}{\flushleft \vspace{-2pt}$a = 0$, $r=-\tau_1^2$, $\tau_1=\tau_2=-\tau_3$}} & $\tau_1 > 0$ & $\mathfrak{su}(1,1)\oplus \mathfrak{su}(3)$ & 1 & \multirow{2}{*}{11} & \multirow{2}{*}{\parbox{2.5cm}{\flushleft Tachibana--Liu\\ solutions \cite{MR0266121}}} \\
         \cline{2-4}
         & $\tau_1 < 0$ & $\mathfrak{su}(2)\oplus\mathfrak{su}(2,1)$ & 1 &  & \\
         \hline
         \multicolumn{2}{|c|}{$a = 0$, $\tau_1=\tau_2=\tau_3$, $r\neq -\tau_1^2$} & $\mathfrak{u}(3)$ & 2 &  & 9-d. isotropy \\
         \cline{1-4}\cline{6-6}
         \multirow{3}{*}{\parbox{3cm}{\flushleft $a_1=a_2=0$, $\tau_1=\tau_2$, $\tau_3=\frac{-2a_3^2+\tau_1(r+\tau_1^2)}{r+\tau_1^2}$, $a_3\neq 0$, $\tau_1\neq\tau_3$, $r\neq-\tau_1^2$}} & \parbox{2.2cm}{\centering $\,$ \\$\tau_1< 0$, $a_3 = \pm \frac{(r+\tau_1^2)}{2\sqrt{-\tau_1}}$} & \parbox{3cm}{\centering $\mathfrak{s}_6\rtimes \mathfrak{su}(2)$ defined by \eqref{eqn: kahler 3fold aut solvable mat algebra} and \eqref{eqn: s-rep of u2}}& 2 & \multirow{7}{*}{9}& \multirow{3}{*}{\parbox{2.4cm}{\flushleft 4-dimensional isotropy}} \\
         \cline{2-4}
         & \vphantom{$\Bigg($}\parbox{2.2cm}{\centering  $a_3^2 < \frac{(r+\tau_1^2)^2}{-4\tau_1}$}& $\mathfrak{u}(2,1)$ & 3 & & \\
         \cline{2-4}
         & \vphantom{$\Bigg($}\parbox{2.2cm}{\centering  $a_3^2 > \frac{(r+\tau_1^2)^2}{-4\tau_1}$}& $\mathfrak{u}(3)$ & 3 & & \\
         \cline{1-4}\cline{6-6}
         \multirow{3}{*}{\parbox{3cm}{\flushleft $a=0$, $\tau_1=\tau_2$, $r=-\tau_1^2$, $\tau_1\neq\tau_3$}} & $\tau_1=0$ & \parbox{3cm}{\centering $\mathfrak{s}_6\rtimes \mathfrak{su}(2)$ defined by \eqref{eqn: kahler 3fold aut solvable mat algebra} and \eqref{eqn: s-rep of u2}} & 1 & & \multirow{3}{*}{\parbox{2.4cm}{\flushleft 5-dimensional isotropy}}\\
         \cline{2-4}
         & $\tau_1>0$ & $\mathfrak{u}(3)$ & 2 & & \\
         \cline{2-4}
         & $\tau_1<0$ & $\mathfrak{u}(2,1)$ & 2 & & \\
         \hline
         \multirow{3}{*}{\parbox{3cm}{\flushleft $a=0$, $\tau_1=\tau_2$, $r=-\tau_1^2$, $\tau_1\neq\tau_3$}} & $\tau_3=0$ & \parbox{3cm}{\centering $\mathfrak{s}_4\oplus \mathfrak{su}(2)$ defined by \eqref{eqn: kahler surface aut solvable mat algebra}} & 1 & & \multirow{3}{*}{\parbox{2.4cm}{\flushleft 5-dimensional isotropy}}\\
         \cline{2-4}
         & $\tau_3>0$ & $\mathfrak{u}(2)\oplus \mathfrak{su}(2)$ & 2 & & \\
         \cline{2-4}
         & $\tau_3<0$ & $\mathfrak{u}(2)\oplus \mathfrak{su}(1,1)$ & 2 & \multirow{5}{*}{7}& \\
         \cline{1-4}\cline{6-6}
         \multirow{2}{*}{\parbox{3.7cm}{\flushleft \vspace{-10pt}$a_2=a_3=0$, $\tau_1\neq\tau_2$, $\tau_2\neq\pm\tau_3$, $\tau_3\neq\tau_1$, $a_1^2 = \frac{-(\tau_1-\tau_2)x}{2}$, $r=x-\tau_2^2$, where $x:=(\tau_1-\tau_3)(\tau_2+\tau_3)$}} & $\tau_1<\tau_2$ &\parbox{2.3cm}{\centering $\,$ \\$\,$ \\ $\mathfrak{u}(2)\oplus \mathfrak{su}(2)$ \\ $\,$}& 3 & & \multirow{2}{*}{\parbox{2.4cm}{\flushleft 2-dimensional isotropy}}\\
         \cline{2-4}
         & $\tau_1>\tau_2$ &\parbox{3cm}{\centering $\,$ \\$\mathfrak{u}(2)\oplus \mathfrak{su}(1,1)$ \\$\,$ \\ $\,$} & 3 & & \\
         \cline{1-4}\cline{6-6}
         \multirow{2}{*}{\parbox{3.2cm}{\flushleft $a=0$, $\tau_1=-\tau_2$, $r=-\tau_1^2$, $\tau_1\neq\pm\tau_3$}} & $\tau_1 = 0$& \parbox{3cm}{\centering $\mathfrak{s}_7$ defined by \eqref{eqn: kahler 3fold aut solvable mat algebra}}& 1 & & \multirow{2}{*}{\parbox{2.4cm}{\flushleft 3-dimensional isotropy}}\\
         \cline{2-4}
         & $\tau_1 \neq 0$& $\mathfrak{u}(2)\oplus \mathfrak{su}(1,1)$ & 2 & &\\
         \hline
    \end{tabular}
    \caption{Strata in the moduli space of germs of Bochner--K\"{a}hler 3-manifolds where the symmetry algebra is constant with dimension $\geq 7$. Strata are described by conditions on the canonicalised moduli space parameters $(a,X,r)$ of Theorem \ref{thm: general potentials} with $\tau_j$ being eigenvalues of $X$. Permuting indices in a strata's defining conditions describes another strata with equivalent properties.
    }
    \label{tab: 3fold moduli space stratification}
\end{table}

\begin{table}[hbt!]
    \centering
    \begin{tabular}{|c:c|c|c:c|l|}
        \hline
         \multicolumn{2}{|c|}{\shortstack[c]{strata defining conditions\\(defined up to\\ permutation of indices)}} & \shortstack[c]{symmetry\\algebra} & \multicolumn{2}{|c|}{\shortstack[c]{strata \&\\ algebra\\ dimensions}}& notes\\\hline\hline
         \multicolumn{2}{|c|}{\parbox{6cm}{$a = 0$, $\tau_1=\tau_2\neq \tau_3$, $r\neq-\tau_1^2$, $r\neq-\tau_3^2$ }} & $\mathfrak{su}(2)\oplus \mathbb{R}^2$ & 3 & \multirow{11}{*}{5} & \parbox{1.5cm}{\flushleft 5-d. isotropy} \\
         \cline{1-4}\cline{6-6}
         \multicolumn{2}{|c|}{$a_2 = a_3 = 0$, $\tau_1=\tau_2=\tau_3$, $a_1\neq 0 $} & $\mathfrak{su}(2)\oplus \mathbb{R}^2$ & \parbox{.7cm}{\centering 2} &  &  \multirow{2}{*}{\parbox{1.5cm}{\flushleft 4-d. isotropy}}\\
         \cline{1-4}
         \multicolumn{2}{|c|}{\parbox{6cm}{$a_1 = a_2 = 0$, $\tau_1=\tau_2\neq \tau_3$, $r=-\tau_1^2$, $a_3\neq 0 $}} & $\mathfrak{su}(2)\oplus \mathbb{R}^2$ & 3 &  &  \\
         \cline{1-4}\cline{6-6}
         \multirow{3}{*}{\parbox{3.3cm}{\vspace{-10pt}\flushleft $a_1^2=2(\tau_2^2+r)(\tau_2-\tau_1)$, $a_2=a_3=0$, $\tau_1\neq \tau_2$, $\tau_2\neq \tau_3$, $\tau_3\neq \tau_1$, $r\neq \tau_1\tau_2+\tau_1\tau_3- \tau_2\tau_3-(\tau_2-\tau_3)^2$}} & \parbox{3.1cm}{\centering \vspace{4pt} $\,$ \\$\frac{{\left(4 \tau_{1} \tau_{2} - 5 \tau_{2}^{2} - r\right)}^{2}}{\left(\tau_{2}^{2} + r\right)^{2}}>0$ $\,$ \\} &$\mathfrak{su}(1,1)\oplus \mathbb{R}^2$ & 4 & & \multirow{6}{*}{\parbox{1.5cm}{\flushleft 2-d. isotropy}}\\
         \cline{2-4}
         & \parbox{3.1cm}{\centering \vspace{4pt}$\,$ \\$\frac{{\left(4 \tau_{1} \tau_{2} - 5 \tau_{2}^{2} - r\right)}^{2}}{\left(\tau_{2}^{2} + r\right)^{2}}<0$ $\,$ \\} &$\mathfrak{su}(2)\oplus \mathbb{R}^2$ & 4 & & \\
         \cline{2-4}
         & \parbox{3.1cm}{\centering \vspace{4pt}$\,$ \\$\frac{{\left(4 \tau_{1} \tau_{2} - 5 \tau_{2}^{2} - r\right)}^{2}}{\left(\tau_{2}^{2} + r\right)^{2}}=0$ $\,$ \\} &\parbox{2.4cm}{\centering $\mathfrak{s}_4\oplus \mathbb{R}$ defined by \eqref{eqn: kahler surface aut solvable mat algebra}}& 3 & & \\
         \cline{1-4}
         \multirow{3}{*}{\parbox{3.3cm}{\flushleft $a=0$, $r=-\tau_1^2$, $\tau_1\neq\tau_2$, $\tau_2\neq\tau_3$, $\tau_3\neq\tau_1$ }} & $\tau_1<0$ &\parbox{2.7cm}{\centering $\,$ \\ $\mathfrak{su}(1,1)\oplus \mathbb{R}^2$ \\ $\,$}& 3 & & \\
         \cline{2-4}
         & $\tau_1>0$ &\parbox{2.7cm}{\centering $\,$ \\ $\mathfrak{su}(2)\oplus \mathbb{R}^2$\vspace{2pt}}& 3 & & \\
         \cline{2-4}
         & $\tau_1=0$ &\parbox{2.4cm}{\centering \vspace{2pt} $\mathfrak{s}_4\oplus \mathbb{R}$ defined by \eqref{eqn: kahler surface aut solvable mat algebra}}\vspace{2pt}& 2 & & \\
         \cline{1-4}\cline{6-6}
         \multirow{3}{*}{\parbox{3.3cm}{\flushleft $a_3=0$, $\tau_1\neq\tau_2$, $\tau_2\neq\tau_3$, $\tau_3\neq\tau_1$, $a_1 \neq 0$, $rx=2a_1^2(\tau_3-\tau_2) +2a_2^2(\tau_3-\tau_1) - \tau_3^2x$, where $x:=(\tau_1-\tau_3)(\tau_2-\tau_3)$}} & $\tau_3>\frac{-a_2^2}{(\tau_2-\tau_3)^2}+\frac{-a_1^2}{(\tau_1-\tau_3)^2}$ &\parbox{2.7cm}{\centering $\,$ \\ $\mathfrak{su}(2)\oplus \mathbb{R}^2$ \\ $\,$}& 5 & & \multirow{3}{*}{\parbox{1.5cm}{\flushleft 1-d. isotropy}}\\
         \cline{2-4}
         & $\tau_3<\frac{-a_2^2}{(\tau_2-\tau_3)^2}+\frac{-a_1^2}{(\tau_1-\tau_3)^2}$ &\parbox{2.3cm}{\centering $\,$ \\ $\mathfrak{su}(1,1)\oplus \mathbb{R}^2$ \\ $\,$ \\ $\,$ }& 5 & & \\
         \cline{2-4}
         & $\tau_3=\frac{-a_2^2}{(\tau_2-\tau_3)^2}+\frac{-a_1^2}{(\tau_1-\tau_3)^2}$ &\parbox{2.4cm}{\centering $\,$ \\ $\mathfrak{s}_4\oplus \mathbb{R}$ defined by \eqref{eqn: kahler surface aut solvable mat algebra} \\ $\,$ }& 4 & & \\
         \hline
         \multicolumn{2}{|c|}{otherwise} & \parbox{1cm}{\centering \vspace{2pt}  $\mathbb{R}^3$ \vspace{2pt}} & 7 & 3 & \\
         \hline
    \end{tabular}
    \caption{Strata in the moduli space of germs of Bochner--K\"{a}hler 3-manifolds where the symmetry algebra is constant with dimension $< 7$. Completes the moduli space stratification when combined with Table \ref{tab: 3fold moduli space stratification}.
    }
    \label{tab: 3fold moduli space stratification low dim symmetries}
\end{table}

Similar to \eqref{eqn: kahler surface aut solvable mat algebra}, the classification involves two solvable algebras $\mathfrak{s}_6$ and $\mathfrak{s}_7$ with $\mathfrak{s}_6\subset \mathfrak{s}_7$, having a matrix representation for which there is a basis $(e_0,\ldots, e_5)$ of $\mathfrak{s}_6$ extending to a basis $(e_0,\ldots, e_6)$ of $\mathfrak{s}_7$ satisfying
\begin{equation}\label{eqn: kahler 3fold aut solvable mat algebra}
\sum_{j=0}^6x_je_j = 
\begin{bmatrix}
0 & 0 & 0 & 0 & 0 & 0 \\
x_1 & 0 & 0 & 0 & x_0 & 0 \\
x_2 & 0 & 0 & x_0-x_6 & 0 & 0 \\
x_3 & 0 & x_6-x_0 & 0 & 0 & 0 \\
x_4 & -x_0 & 0 & 0 & 0 & 0 \\
x_5 & -x_4 & -x_3 & x_2 & x_1 & 0 
\end{bmatrix}.
\end{equation}
We will also encounter the $\mathfrak{s}_6$-representation $\rho:\mathfrak{su}(2)\to \mathfrak{aut}(\mathfrak{s}_6)$ of the unitary algebra $\mathfrak{su}(2)$, represented by matrices in the basis $(e_0,\ldots, e_5)$ of the form
\begin{equation}\label{eqn: s-rep of u2}
\begin{bmatrix}
0 & 0 & -u_2 & u_3 & -u_1 & 0 \\
0 & 0 & u_2 & -u_3 & u_1 & 0 \\
0 & -u_2 & 0 & -u_1 & -u_3 & -u_2 \\
0 & u_3 & u_1 & 0 & -u_2 & u_3 \\
0 & -u_1 & u_3 & u_2 & 0 & -u_1 \\
0 & 0 & 0 & 0 & 0 & 0 
\end{bmatrix}.
\end{equation}
Their corresponding semidirect sum $\mathfrak{s}_6\rtimes \mathfrak{su}(2)$ appears among the symmetry algebras in Table \ref{tab: 3fold moduli space stratification}.

We omit full details of deriving Tables \ref{tab: 3fold moduli space stratification} and \ref{tab: 3fold moduli space stratification low dim symmetries} as the calculations are prohibitively long, while offering little illumination beyond what is already demonstrated in \S\ref{complex surface case}. We will, however, present the special case covered by rows 7 through 9 in Table \ref{tab: 3fold moduli space stratification low dim symmetries}, as this is a good demonstration of the basic techniques involved. 

This special case consists of the Bochner--K\"{a}hler structures whose $(a,X,r)$ parameters in the normal form of Theorem \ref{thm: general potentials} satisfy
\begin{equation}\label{eqn: demonstration strata}
a=0,
\quad X = \begin{bmatrix} \tau_1 &0 &0\\ 0&\tau_2 &0\\0 &0&\tau_3 \end{bmatrix},
\quad r=-\tau_1^2,
\quad \tau_1\neq\tau_2,
\quad \tau_2\neq\tau_3,
\quad\text{and}\quad \tau_3\neq\tau_1.
\end{equation}
The algebra of infinitesimal symmetries described in Theorem \ref{thm: infinitesimal symmetries}, represented by matrices in the form \eqref{indef su rep symmetry}, is spanned by
\[
e_1:=\begin{bmatrix} 
0 & -2\tau_1 & 0 & 0 & 0 \\ 
1 & 0 & 0 & 0 & \I\tau_1 \\
0 & 0 & 0 & 0 & 0 \\
0 & 0 & 0 & 0 & 0 \\
0 & 2\I & 0 & 0 & 0 \\
\end{bmatrix},
\quad
e_2:=\begin{bmatrix} 
0 & 0 & 0 & 0 & 0 \\ 
0 & 0 & 0 & 0 & 0 \\
0 & 0 & 0 & 0 & 0 \\
0 & 0 & 0 & \I & 0 \\
0 & 0 & 0 & 0 & 0 \\
\end{bmatrix},
\quad
e_3:=\begin{bmatrix} 
0 & 2\I\tau_1 & 0 & 0 & 0 \\ 
\I & 0 & 0 & 0 & -\tau_1 \\
0 & 0 & 0 & 0 & 0 \\
0 & 0 & 0 & 0 & 0 \\
0 & 2 & 0 & 0 & 0 \\
\end{bmatrix},
\]
\[
e_4:=\begin{bmatrix} 
0 & 0 & 0 & 0 & 0 \\
0 & \I & 0 & 0 & 0 \\
0 & 0 & 0 & 0 & 0 \\
0 & 0 & 0 & 0 & 0 \\
0 & 0 & 0 & 0 & 0 \\
\end{bmatrix},
\quad\text{and}\quad
e_5:=\begin{bmatrix} 
0 & 0 & 0 & 0 & 0 \\
0 & 0 & 0 & 0 & 0 \\
0 & 0 & \I & 0 & 0 \\
0 & 0 & 0 & 0 & 0 \\
0 & 0 & 0 & 0 & 0 \\
\end{bmatrix}
\]
modulo the identity matrix and the Reeb field, which is represented by \eqref{eqn: potential parameters}. 

The algebra's Killing form represented with respect to this basis is
\[
\begin{bmatrix}-32 \tau_{1} & 0 & 0 & 0 & 0 \\ 0 & 0 & 0 & 0 & 0 \\ 0 & 0 & -32 \tau_{1} & 0 & 0 \\ 0 & 0 & 0 & -2 & 0 \\ 0 & 0 & 0 & 0 & 0\end{bmatrix},
\]
which reveals that we have at least $3$ isomorphism classes among these algebras prameterised by $(\tau_1,\tau_2,\tau_3)$, as the Killing form's inertia is an invariant. Hence we must consider $\tau_1<0$, $\tau_1=0$, and $\tau_1>0$ separately. 

If $\tau_1\neq 0$, the subspace $\mathfrak{a}:=\langle e_1,e_3,4\tau_1 e_4+(\tau_1 + \tau_3)e_2+(\tau_1+\tau_2)e_5\rangle$ is closed under Lie brackets and its complement $\langle e_2, e_5\rangle$ is the algebra's center. The killing form of $\mathfrak{a}$ is either negative definite or mixed signature depending on whether $\tau_1>0$ or $\tau_1<0$. Hence $\mathfrak{a}$ is $\mathfrak{su}(2)$ or $\mathfrak{su}(1,1)$ when $\tau_1\neq0$ depending on its sign. When $\tau_1=0$, $e_5$ is in the algebra's center, whereas the $\langle 4\tau_3e_2+4\tau_2e_5, e_1, e_3, e_4\rangle$ is isomorphic to the matrix algebra in \eqref{eqn: kahler surface aut solvable mat algebra}, visible immediately because these two algebras have identical structure equations in the respective bases that we have presented them. Thus, there are exactly three isomorphism classes of symmetry algebras in the moduli space stratum \eqref{eqn: demonstration strata}.

\section{Acknowledgements}
The authors wish to thank Duong Ngoc Son for inspiring discussions, as well as Owen Dearricott for useful discussions and for drawing our attention to related results by Pan\'ak and Schwachh\"{o}fer.
The third author was supported by the Institute for Basic Science (IBS-R032-D1).

\end{document}